\documentclass[nonblindrev]{informs3} 
\renewcommand{\theARTICLETOP}{}
\OneAndAHalfSpacedXI 

\usepackage{amsfonts,amscd,epsfig}
\usepackage{amsmath}
\usepackage{mathtools}
\usepackage[colorlinks = true,
            linkcolor = black,
            urlcolor  = black,
            citecolor = black,
            anchorcolor = blue]{hyperref}
\usepackage{amsfonts,amscd,epsfig}
\usepackage{graphicx}
\usepackage{natbib}
\usepackage[small, margin=1cm]{caption}
\usepackage{subfigure}
\usepackage{bbm}
\usepackage{tabularx}
\usepackage[linesnumbered,ruled]{algorithm2e}
\usepackage{bm}
\usepackage{diagbox}
\usepackage{booktabs}
\usepackage{multirow}
\usepackage{geometry}
\usepackage{ltxtable}
\usepackage{hyperref}
\usepackage{booktabs,caption}
\usepackage{flafter} 
\usepackage{algpseudocode}
\usepackage{bbm}
\usepackage[flushleft]{threeparttable} 

\usepackage{siunitx}
\usepackage{cleveref}
\usepackage{makecell}

\renewcommand{\emptyset}{\varnothing}

\newcommand{\bY}{\mathcal{Y}}
\newcommand{\Ex}{\mathbb{E}}

\newcommand{\R}{\mathbb{R}}

\allowdisplaybreaks

\usepackage{natbib}
\bibpunct[, ]{(}{)}{,}{a}{}{,}%
\def\bibfont{\small}%

\usepackage{rotating}
\usepackage{fancyvrb}

\TheoremsNumberedThrough     
\ECRepeatTheorems

\EquationsNumberedThrough    

\begin{document}
	
	
	\RUNAUTHOR{Zhu et al.}
	
	\RUNTITLE{Huge-Scale Assortment Optimization with Customer Choice: A Parallel Primal-Dual Approach}
	
	\TITLE{Huge-Scale Assortment Optimization with Customer Choice: A Parallel Primal-Dual Approach}
	\ARTICLEAUTHORS{
       \AUTHOR{Donghao Zhu}
		\AFF{The Institute of Statistical Mathematics, Tokyo, Japan 190-0014, \\
        The University of Tokyo Market Design Center, The University of Tokyo, Tokyo, Japan 113-0033, \\
        \EMAIL{zhu@ism.ac.jp}}
        \AUTHOR{Hanzhang Qin}
        \AFF{
        Institute of Operations Research and Analytics, National University of Singapore, Singapore 117602,\\
        Department of Industrial Systems Engineering and Management, National University of Singapore, Singapore 117576,
        \EMAIL{hzqin@nus.edu.sg}}
         \AUTHOR{Ching-pei Lee}
		\AFF{
        The Institute of Statistical Mathematics, Tokyo, Japan 190-0014, \EMAIL{chingpei@ism.ac.jp}}
         \AUTHOR{Yuki Saito, Takahiro Kawashima}
		\AFF{
        ZOZO Next, Inc., Tokyo, Japan 263-0023,  \EMAIL{yuki.saito@zozo.com}, \EMAIL{takahiro.kawashima@zozo.com}}
         \AUTHOR{Kenji Fukumizu}
		\AFF{
        The Institute of Statistical Mathematics, Tokyo, Japan 190-0014, \EMAIL{fukumizu@ism.ac.jp}}
	} 
	
\ABSTRACT{We study huge-scale assortment optimization problems to maximize expected revenue under customer choice, addressing a fundamental challenge in industries such as transportation, retail, and healthcare.
The choice-based linear programming (CBLP) formulation provides a powerful framework for optimizing sales allocations across customer segments, yet traditional approaches often fail to solve CBLPs of huge scale (involving millions of customer choices) due to the lack of algorithmic designs that exploit problem structure.
To overcome this computational bottleneck, we propose a first-order primal–dual method, \texttt{SPFOM}, which requires only a small computational cost per iteration, achieves a provably near-optimal convergence rate, and can be readily extended to parallel computing environments.
Computational experiments demonstrate the computational and practical superiority of \texttt{SPFOM} over state-of-the-art solvers for large-scale linear programs.
The framework is extended to a multi-period assortment optimization setting with inventory constraints, where \texttt{SPFOM} estimates global shadow prices that enhance classical bid-price control policies compared with benchmark methods such as market segment decomposition.
Numerical experiments and a case study using real-world data from the ZOZOTOWN platform validate the practical effectiveness of \texttt{SPFOM}, highlighting its advantages in improving revenue performance while maintaining balanced inventory levels.}


	
	\maketitle
    
\section{Introduction}

In the dynamic landscape of modern digital markets, effective revenue management (RM) is crucial for businesses to thrive.
This paper investigates the computational challenges arising in huge-scale assortment optimization problems (i.e., extremely large-scale problems involving millions of customer choices and high-dimensional decision spaces) that aim to maximize expected revenue, with a particular focus on formulations based on choice-based linear programs (CBLPs) under the generalized attraction model (GAM).
A central question in revenue management is how to incorporate demand dependencies into forecasting and optimization models. 
Specifically, suppliers must anticipate how consumer demand redistributes when the set of available products changes, since such substitution effects fundamentally determine the revenue-maximizing assortment. 
Modeling these interactions is critical for effective decision-making in large, dynamic markets where customer choices evolve in response to assortment updates and price adjustments.
This modeling framework offers a substantial improvement over the traditional independent demand model (IDM), which relies on the simplifying assumption that the demand for each product is independent of the availability of other products. 
While the IDM allows for tractable computation, it fails to capture the cross-product substitution effects that are central to realistic demand systems.

In this paper, we focus on the discrete choice model and the associated optimization formulation developed by \citet{gallego2015general}.
Specifically, the GAM provides a closed-form representation of customer demand for any given choice set, capturing substitution effects across products.
The parameters of the GAM can be efficiently estimated from abundant transaction data \citep{gallego2015general}, and are assumed to be given in this study.
Under the GAM framework, the resulting revenue management problem can be formulated as a choice-based linear program (CBLP), which can be equivalently transformed into a sales-based linear program (SBLP) that is more amenable to large-scale computation and algorithmic analysis.

In modern digital markets, assortment optimization problems often involve an extraordinarily large number of decision variables, reflecting the complexity of coordinating millions of products, heterogeneous customer segments, and dynamic demand patterns.
For instance, in the fashion e-commerce industry, platforms such as ZOZOTOWN, one of Japan’s largest online fashion retailers, handle tens of millions of product–customer interactions across thousands of brands and millions of active users daily~\citep{saito2020open}.
Similarly, Amazon recorded over 3 billion monthly online visits in 2023, illustrating the massive scale of customer activity in global digital markets~\citep{amazon2024sell}.
Managing such large-scale systems presents formidable computational challenges, as even basic full-dimensional vector operations become prohibitively expensive at realistic market scales.
A common practical approach to mitigate this challenge is to decompose the problem by partitioning the product space into smaller, more manageable subsets (e.g., market segments), and making assortment or selling decisions separately for each subset~\citep{aouad2023market}.
However, while such decomposition strategies reduce the dimensionality of the product space, the number of customers and potential choice interactions entering the system typically remains prohibitively large.
These characteristics make the development of computationally scalable optimization methods indispensable for practical large-scale assortment optimization.
Hence, our central research question is:
\textit{
How can we design a computationally efficient method to solve assortment optimization problems when the number of customers is prohibitively large?
}
To answer this question, the primary objective of this paper is to develop a computationally scalable framework that enables firms to determine more effective resource allocation and assortment (or product recommendation) strategies for large-scale digital markets, while maintaining minimal computational overhead.

\subsection{Our Contributions}
The core technical innovation of our work lies in the development of a \emph{novel primal–dual first-order method}, termed the Sales-based Primal-dual First-Order Method (\texttt{SPFOM}), which efficiently exploits the structured coupling inherent in huge-scale assortment optimization problems and ensures rapid convergence to the optimal solution of the SBLP.
The optimal solution of the SBLP can be equivalently transformed into that of the CBLP, thereby linking the sales-based and choice-based perspectives in a manner consistent with practical modeling considerations.
We establish a global linear contraction in expected squared distance to the unique smoothed optimum, with an explicit rate that depends on the market size, batch size, and the spectral properties of the structured coupling matrix.
This form of convergence is not implied by existing LP theory papers: these papers usually apply to deterministic full-information iterations and control residuals for the original LP, whereas our result quantifies how sampling noise, problem scale, and smoothing jointly determine contraction for a market-structured problem.

Moreover, from a computational perspective, \texttt{SPFOM} is inherently parallelizable, enabling substantial acceleration in distributed computing environments using GPUs. 
Our algorithm contributes to the growing literature on \emph{GPU-based linear programming}, and our approach is more structured and scalable than general-purpose GPU-based LP methods, as it explicitly exploits the problem structure arising in huge-scale assortment optimization, as demonstrated in our numerical experiments.
Our framework can also be extended to \emph{multi-period assortment optimization} settings with inventory considerations, allowing it to adapt effectively to dynamically evolving retail environments.

The global shadow prices estimated by \texttt{SPFOM} can be directly incorporated into classical revenue management heuristics, such as bid-price control.
In contrast to conventional approaches that repeatedly compute local shadow prices for smaller product subsets or market segments, \texttt{SPFOM} derives global shadow prices that more accurately capture product-level marginal profits, substitution effects, and future value potentials.
This capability enables the system to avoid prematurely recommending products with higher long-term value, thereby improving both revenue performance and inventory balance over multiple periods.
This framework can also be extended to network revenue management problems with bundle selling, in which bundle-level decisions interact through shared resource capacities.
Using a \emph{real-world dataset} from ZOZOTOWN, one of Japan’s largest online retail platforms, we demonstrate the scalability and empirical effectiveness of the proposed method in large-scale, dynamically evolving assortment and recommendation environments.

\subsection{Organization of the Paper}
Section~\ref{sec2:relate} reviews the related literature.
Section~\ref{sec3:model} introduces the model setup and problem formulations.
Section~\ref{sec4:mothodology} presents our methodological framework and theoretical analysis, together with numerical validation and discussion of parallel implementation.
Section~\ref{sec5:multiNRM} extends the framework to a multi-period assortment optimization setting and shows how the resulting shadow prices improve bid-price control.
Section~\ref{sec6:case} reports a real-world case study using data from ZOZOTOWN and evaluates computational performance across large-scale and multi-period settings.
Section~\ref{sec7:conclusion} concludes and outlines directions for future research.

\section{Related Work}\label{sec2:relate}

A choice-based deterministic linear program (CBLP) is a mathematical model for maximizing revenue through optimal inventory allocation across market segments while explicitly accounting for customer choice behavior \citep{liu2008choice}, and serves as a fundamental formulation for assortment optimization problems under customer choice.
Unlike traditional models that treat product demands as independent, the CBLP incorporates the sale probability of each product within a market segment, capturing interactions among available alternatives and customer preferences.
Its objective is to determine the optimal assortment of products and the fraction of time each assortment should be offered, thereby maximizing expected revenue over the planning horizon.
This deterministic linear programming formulation provides an upper bound on achievable revenue while respecting inventory constraints and capturing customer choice dynamics.
For a comprehensive overview of recent studies on choice-based revenue management and assortment optimization, readers are referred to the surveys by \citet{feng2022consumer, heger2024assortment}.

Extensive research has studied revenue management from both modeling and computational perspectives; for a comprehensive review, readers are referred to \citet{strauss2018review}.
Much of this literature focuses on developing rich optimization formulations to capture demand uncertainty, customer choice behavior, and inventory constraints, as well as tractable solution methods for the resulting models.
Within this line of work, \citet{gallego2015general} introduced the sales-based linear program (SBLP), which is equivalent to the choice-based linear program (CBLP) but avoids the exponential number of variables inherent in the latter.
Relatedly, \citet{tong2014approximate}, \citet{vossen2015reductions}, and \citet{kunnumkal2019choice} proposed tractable approximations to choice-based linear programs under both independent and discrete choice demand models, while \citet{zhang2022product} developed a product-based separable piecewise linear approximation that yields tighter relaxations.
Despite these advances, existing approaches are not directly applicable to the huge-scale regimes we consider, those involving extremely large customer populations and product spaces, as they typically lead to formulations whose computational complexity grows rapidly with problem size.
In contrast, our work focuses on the development of computationally scalable algorithms for large-scale assortment optimization under customer choice, with an emphasis on methods that scale linearly with the number of customers within the CBLP/SBLP framework.

Naturally, our approach belongs to the class of first-order methods for solving large-scale linear programs (LPs).
To date, the most widely used LP solvers have been based on either the simplex method or the interior-point method \citep{HansRanking}.
Both approaches rely heavily on solving linear systems through matrix factorization, which makes them unsuitable for extremely large-scale problems, as Hessian computation and factorization can become prohibitively expensive. Recent studies have demonstrated that first-order methods can generate highly accurate LP solutions in a relatively short time \citep{applegate2021practical}.
By employing first-order techniques, the computational bottlenecks inherent in simplex and interior-point methods can be overcome, as the primary operation reduces to gradient evaluation rather than Hessian computation. 
This shift enables the efficient generation of high-accuracy solutions even at massive problem scales. 
There is a recent trend of studying how to efficiently incorporate GPUs in first-order methods, with a special focus on LP and conic programming solvers (\citealt{lu2025cupdlp, lu2025practical, lin2025pdcs}). 
Our method differs from these general-purpose approaches in that it explicitly exploits the problem structure arising in huge-scale assortment optimization under customer choice, leading to substantially improved computational efficiency compared with generic methods.

Recent efforts have been made to establish the theoretical validity of first-order methods for solving general huge-scale LPs, including the works of \citet{nesterov2008rounding}, \citet{nesterov2012efficiency}, and \citet{lu2021nearly}.
\citet{guo2025solving} proposed ML-based first-order methods for directly optimizing assortments under complex choice models, with a focus on scalability and benchmark performance.
Our approach differs from these studies in that we aim to develop a computationally tractable first-order method tailored to LPs with a specific structure.
Prior research has shown that for LPs possessing certain structural properties, specialized first-order algorithms can achieve faster convergence rates than general-purpose LP solvers, for example, nearly linear-time solvers for Packing LPs \citep{allen2019nearly} and efficient solvers for Optimal Transport LPs \citep{mai2021fast, liao2022fast}.
Nevertheless, our approach again departs from these directions, as it is straightforward to verify that neither the CBLP nor the SBLP falls into the category of Packing or Optimal Transport LPs.

\section{Model Setup and Problem Formulation}\label{sec3:model}

We introduce a stochastic assortment optimization formulation under customer choice.
Let $[k]$ denote the set $\{1, .., k\}$ for any positive integer $k$.
On the demand side, the positive integer $n$ denotes the number of customer types, assumed to be \textit{very large}. 
For each customer type $i \in [n]$, where each type corresponds to an individual customer, we assume a homogeneous arrival rate $\lambda_i=1$.
On the supply side, the positive integer $m$ denotes the number of product types, where $m \ll n$.
For each product type $j \in [m]$, the positive real number $c_j$ denotes the capacity, representing the maximum quantity available for purchase. The corresponding price is given by the positive real number $r_j$.

The set of all potential assortments indicating a customer's purchasing interest is denoted by $\mathcal{S} := 2^{[m]}$.
Each customer, given an assortment $S$, will select product $j\in S$ with probability $\pi_{ij}(S)$. 
Let the consideration sets be distinct across customer types, i.e., customer of type $i$ may choose from a set $\mathcal{S}_i$.
Since the analysis and algorithm remain unchanged, all sets are assumed equal, $\mathcal{S}_i=\mathcal{S}$, for notational simplicity.
For $i\notin S$, $\pi_{ij}(S)=0$. 
Then, upon each customer's arrival, an assortment is made to the customer if a purchase is made for product $j$, revenue $r_j$ is incurred and 1 unit of the inventory for product $j$ is consumed. The RM problem is then how to decide the offering of assortment to all customer arrivals to optimize the expected cumulative revenue. 

Again, for notational simplicity, we adopt the multinomial logit model (MNL) to model the customer choice. Between each customer type $i \in [n]$ and each product type $j \in [m]$, the real number $w_{ij}$ denotes the preference weight. For each customer $i$, the real number $w_{i0}$ denotes the preference that the customer does not want to buy any product. Then, the MNL model specifies
\[
\pi_{ij}(S) = \frac{w_{ij}}{w_{i0}+\sum\limits_{j'\in S } w_{ij'} }.
\]

The MNL model belongs to the basic attraction model (BAM) model \citep{luce2012individual}, but our algorithm design and analysis can easily extend to the case where the choice model is a GAM. 
In that case, besides the nominal attraction values $\{w_{ij}\}$, we use shadow attraction values $\{v_{ij}\}$ to estimate the effect of items that are not in the assortment on customer choice. Namely, the choice probability of customer $i$ for product $j$ becomes
\[
\pi_{ij}'(S) = \frac{w_{ij}}{w_{i0}+\sum\limits_{j'\in S } w_{ij'}+ \sum\limits_{j'\notin S } v_{ij'} }.
\]
Since the equivalence between the CBLP and the SBLP holds not only for the MNL model but also for the broader GAM class \citep{gallego2015general}, the algorithmic design and theoretical guarantees presented in this work naturally extend to all choice models covered by GAM.

\subsection{CBLP}
The choice-based model refers to a decision-making model in which individuals make choices from a set of available alternatives, and it is well known that a deterministic LP formulation provides an upper bound for the RM objective \citep{talluri2004revenue}. 
Our primary objective is to solve the following choice-based LP, defined as:
{
\begin{align}
\max_{\substack{x_i(\mathcal{S}) \\ i\in [n],\, S\in \mathcal{S}}} &\sum_{i\in [n]}\sum_{S\in \mathcal{S}}x_i(S) \cdot \sum_{j\in S} \frac{r_j w_{ij}}{w_{i0}+\sum_{j'\in S} w_{ij'}} \label{eq:choice-based_obj} \\[4pt]
\text{s.t.}\quad
    & \sum_{i\in [n]}\sum_{S\in \mathcal{S}}x_i(S)\mathbf{1}(j \in S) \frac{ w_{ij}}{w_{i0}+\sum_{j' \in S}w_{ij'}} \le c_j, && \forall j\in [m] \label{eq:choice-based_1} \\[2pt] 
    & \sum_{S\in \mathcal{S}}x_i(S) = \lambda_i, && \forall i\in [n] \label{eq:choice-based_2} \\[2pt]
    & x_i(S)\ge 0, && \forall i\in [n],~S\in \mathcal{S}. \label{eq:choice-based_3}
\end{align}
}
\noindent
Here, the function $x_i(S)$ represents the total time that customer $i$ is offered the set $S$ and is the decision variable.
The objective~\eqref{eq:choice-based_obj} maximizes the expected revenue by multiplying the product prices with the probabilities that customers desire to buy those products.
Constraint~\eqref{eq:choice-based_1} specifies that the total quantity of each product purchased by all customers should not exceed the capacity.
Constraint~\eqref{eq:choice-based_2} ensures that the number of customers purchasing products is equal to the arrival rate of each customer type.
Constraint~\eqref{eq:choice-based_3} ensures that customers may only purchase a positive number of products.
This choice-based model has an exponential number of variables due to the size of $\mathcal{S}$.

\subsection{SBLP}
\citet{gallego2015general} developed a sales-based LP by applying the primal-dual theorem to the CBLP (and showed CBLP and SBLP are equivalent), defined as follows:
{
\begin{align}
\max_{\substack{y_{ij} \\ i\in [n] \\ j\in\{0\}\cup[m]}}
    & \sum_{i\in [n]}\sum_{j\in [m]} y_{ij} r_j \label{eq:sales-based_obj} \\[4pt]
\text{s.t.}\quad
    & \sum_{i\in [n]} y_{ij} \le c_j, && \forall j\in [m] \label{eq:sales-based_1} \\[2pt]
    & \sum_{j=0}^m y_{ij} = \lambda_i, && \forall i\in [n] \label{eq:sales-based_2} \\[2pt]
    & \frac{y_{ij}}{w_{ij}} \le \frac{y_{i0}}{w_{i0}}, && \forall i\in [n],~\forall j\in [m] \label{eq:sales-based_3} \\[2pt]
    & y_{ij} \ge 0, && \forall i\in [n],~\forall j\in [m]. \label{eq:sales-based_4}
\end{align}
}
\noindent
Here, the positive real variable $y_{ij}$ denotes the quantity desired by customer $i \in [n]$ for product $j \in [m]$. 
The positive real variable $y_{i0}$ denotes the quantity that customer $i \in [n]$ does not wish to purchase any product. The objective~\eqref{eq:sales-based_obj} maximizes the expected revenue by multiplying the product prices with the quantities that customers desire to purchase.
Constraint \eqref{eq:sales-based_1} delineates the capacity, ensuring that the quantities of products that customers purchase do not exceed their respective capacities. 
Constraint \eqref{eq:sales-based_2} delineates the balance, ensuring that the number of customers of the same type, regardless of whether they purchase a product, equals the corresponding arrival rate.
A \textit{demand} is defined as the ratio of the quantity of products purchased and customer preferences~\citep{gallego2015general}.
Constraint~\eqref{eq:sales-based_3} compares the demand for purchase and no-purchase behaviors.
Constraint~\eqref{eq:sales-based_4} ensures that customers only purchase a positive number of products.
For further details on the derivation process of the sales-based LP, interested readers are referred to~\citet{gallego2015general}.

The SBLP is equivalent to the CBLP, but solving the CBLP is computationally challenging due to the exponential number of decision variables. 
Solving the SBLP is easier for this reason. 
Additionally, the formulation of the SBLP is similar to the Optimal Transport (OT) problem~\citep{cuturi2013sinkhorn, peyre2019computational}, but with an additional constraint~\eqref{eq:sales-based_3}. 
Therefore, solving the SBLP is technically more challenging than solving a general OT problem. 
Our goal is to tackle a huge-scale SBLP, where the parameter $n$ is much larger than $m$.

\section{Methodology Framework and Analysis}\label{sec4:mothodology}

Our methodological framework adopts the BAM model to capture customer choice behavior in an assortment optimization setting, from which the CBLP is derived to compute the optimal product recommendation scheme, as illustrated in Figure~\ref{Fig:illus_1}.
\begin{figure}[tbp]
\caption{\raggedright Framework of our modeling and algorithmic approach. }
    \centering
\includegraphics[width=120mm]{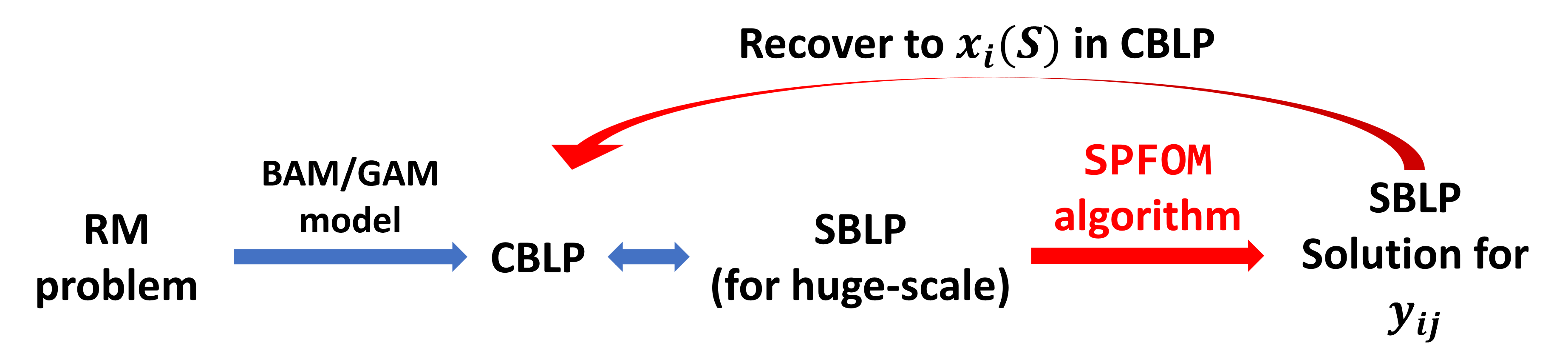}
\label{Fig:illus_1}
\end{figure}
To address the exponential growth of variables in the CBLP, an equivalent formulation, termed the SBLP, is identified, making large-scale optimization computationally tractable.
In this section, we present our method for solving the SBLP, which can be viewed as a primal-dual first-order algorithm with sparse primal updates, referred to as \texttt{SPFOM}.
The solution to the SBLP obtained by \texttt{SPFOM} can be equivalently transformed into the solution of the CBLP.
The potential for parallel implementation within this framework is discussed.
The complete structure of the \texttt{SPFOM} algorithm and its parallelized version are provided in Algorithms~\ref{alg:sda_non} and~\ref{alg:sda}.

\subsection{The Primal-Dual Approach}
The constraint~\eqref{eq:sales-based_1} is incorporated into the objective function~\eqref{eq:sales-based_obj} to obtain:
{
\begin{align}
\min_{\substack{\eta_j \\ j\in [m]}} 
    \max_{\substack{y_{ij} \\ i\in [n] \\ j\in \{0\}\cup[m]}}
    & \sum_{i\in [n]}\sum_{j\in [m]} y_{ij} r_j + \sum_{j\in [m]}\eta_j \!\left(c_j - \sum_{i\in [n]} y_{ij}\right)
      \label{entro_obj} \\[4pt]
\text{s.t.}\quad
    & \sum_{j=0}^m y_{ij} = \lambda_i, && \forall i\in [n] \label{entro_cons_1} \\[2pt]
    & \frac{y_{ij}}{w_{ij}} \le \frac{y_{i0}}{w_{i0}}, && \forall i\in [n],~\forall j\in [m] \label{entro_cons_2} \\[2pt]
    & y_{ij} \ge 0, && \forall i\in [n],~\forall j\in [m] \label{entro_cons_3} \\[2pt]
    & \eta_j \ge 0, && \forall j\in [m]. \label{entro_cons_4}
\end{align}
}
\noindent
This formulation makes fast and sparse gradient updates on $i$.
Our primary focus is on fast updates for the inner optimization problem (recall that $n$ is assumed to be very large, so even scanning an $n$-dimensional vector can be computationally costly), as the dual updates can be executed using a standard dual ascent approach.
This assumption relies on the premise that $m$ is not excessively large.
Namely, the inner problem for customer $i$ is:
{
\begin{align*}
z_i := 
\max_{\substack{y_{ij} \\ j\in [m]}}
    & \sum_{j\in [m]} y_{ij}(r_j - \eta_j) \\[4pt]
\text{s.t.}\quad
    & \sum_{j=0}^m y_{ij} = \lambda_i, \\[2pt]
    & \frac{y_{ij}}{w_{ij}} \le \frac{y_{i0}}{w_{i0}}, && \forall j\in [m] \\[2pt]
    & y_{ij} \ge 0, && \forall j\in [m].
\end{align*}
}
\noindent
A key observation regarding the structure of the inner optimization problem is, a greedy approach in $O(m)$ is sufficient to solve the problem exactly when the value of $y_{i0}$ is given.

\begin{algorithm}
\OneAndAHalfSpacedXI
\caption{\texttt{FAST-INNER}}\label{alg:fast inner}
\textbf{Input:} Customer index $i$ and non-purchase quantity $y_{i0}$\\
\textbf{Initialization:} Set $\bar \lambda_i \leftarrow \lambda_i$, $t\leftarrow 1$, $\forall~j=1,\ldots,m:y_{ij}\leftarrow 0$\\
\textbf{Sorting:} Sort $(r_j-\eta_j)$ in descending order and denote $y_{i[j]}$ as the variable that has the $j$th largest coefficient $(r_{[j]}-\eta_{[j]})$.\\
\textbf{While} $\bar \lambda_i > 0$:\\
\quad Set $y_{i[t]} = \min\left\{ \frac{w_{i[t]}y_{i0}}{w_{i0}},\bar \lambda_i \right\}$\\
\quad $\bar \lambda_i \leftarrow \bar \lambda_i - y_{i[t]}$\\
\quad $t \leftarrow t + 1$\\
\textbf{End while} \\
\textbf{Output:} $\{y_{ij}\}$ and $z_i = \sum\limits_{j\in [m]}y_{ij}(r_j-\eta_j)$.
\end{algorithm}

\begin{lemma}\label{lemma_1}
    When $y_{i0}$ is fixed, $z_i$ can be computed by a greedy algorithm (Algorithm \ref{alg:fast inner}) with computational time $O(m)$.
\end{lemma}

\noindent
The proof of Lemma~\ref{lemma_1} is provided in Online Appendix~\ref{Proof:Lemma1}.
In practice, a mini-batch of customers can be sampled, and each customer $i$'s inner optimization can be efficiently solved via a golden-section search over the value of $y_{i0}$.
This approach requires only $\tilde O(mB)$ time, where logarithmic factors are ignored and $B$ is the batch size. 
The Algorithm~\ref{alg:fast inner} solves the inner problem $z_i$ in this way.

After solving the inner problem, we perform the dual ascent update.
For each $j \in [m]$, $\eta_j$ is updated as follows:
{
\begin{equation*}
\eta_j \leftarrow \max\left\{ 0, \eta_j - \tau\left( c_j -  \sum_{i\in [n]}  y_{ij} \right)  \right\},
\end{equation*}  
}
where $\tau>0$ is a constant step size. 
The complete \texttt{SPFOM} algorithm is summarized in Algorithm~\ref{alg:sda_non}.

\begin{algorithm}[tbp]
\OneAndAHalfSpacedXI
\caption{\texttt{SPFOM}}\label{alg:sda_non}
\textbf{Input:} Step size $\tau$, number of iterations $T$, batch size $B$, initial values $\{y_{ij}^0\}$, $\{\eta_j^{0}\}$\\
\textbf{For} $t=1,2,\ldots,T$: \\
\quad Randomly generate a subset of customers $\mathcal{I} \subset [n]$ with size $B$\\
\quad Solve the inner optimization problem $z_i$ and update $y_{ij}^t$ for all $i \in \mathcal{I}$ and $j\in [m]$ through golden-section search and \texttt{FAST-INNER}\\
\quad Execute the projected dual ascent update:\\
\quad \quad \textbf{For} \( j \in [m] \):\\
\quad \quad \quad \( \eta_j^t \leftarrow \max \left\{ 0, \eta_j^{t-1} - \tau \left( c_j - \sum\limits_{i \in [n]} y_{ij} \right) \right\} \)\\
\quad \quad \textbf{End for}\\
\textbf{End for}\\
\textbf{Output:} $\{y_{ij}^T\}$, $\{\eta_j^{T}\}$ and the objective value.
\end{algorithm}

\noindent
The computational cost of each iteration of \texttt{SPFOM} is $\tilde O(mB)$ by design. 
This efficient primal-dual first-order method with sparse primal updates ensures that our approach remains computationally feasible even for huge-scale instances.

\subsection{Optimality Guarantees}
Convergence guarantees are provided for a slightly modified variant of the algorithm, whose solution is shown to be arbitrarily close to that of the original formulation in theory.
In particular, consider the case where a small regularization term, $(\mu/2)\sum\limits_{j\in[m]}\eta_j^2$, is added to the minimax formulation.
This is equivalent to solving the following penalized version of \eqref{eq:sales-based_obj}, where the constraint set for $y$ in the original minimax formula is denoted as $Y$.
{\small
\begin{align}
    \max\limits_{\substack {y_{ij}\in Y \\ i\in [n] \\ j\in\{0\} \cup [m]}}  & \sum_{i\in [n]}\sum_{j\in [m]}y_{ij}r_j + \frac{1}{\mu}\sum_{j: \sum\limits_{i \in [n]} y_{ij} > c_j} \left(\sum\limits_{i\in [n]}y_{ij}\ - c_j\right)^2 \label{eq:new_obj}.
\end{align}
}
Accordingly, the dual descent step in Algorithm~\ref{alg:sda_non} should be modified to
{\small
\begin{equation}\label{eq:stronglycvxupdate}
    \eta_j^t = \max\left\{ 0, (1 - \tau\mu) \eta_j^{t-1} - \tau\left( c_j -  \sum\limits_{i\in [n]}  y_{ij} \right)  \right\}.
\end{equation}
}

\noindent
The following theorem, from Theorem~17.1 of \citet{NocW06a}, shows that as $\mu$ approaches zero, the solution of \eqref{eq:new_obj} recovers the true solution of the original linear programming problem.
\begin{theorem}[Optimality Preservation]\label{Theo_1_old}
Let $\bar y_{\mu}$ be the optimal solution of \eqref{eq:new_obj}. For any sequence $\{\mu_i\}$ that approaches $0$, we have that all the limit points of $\bar y_{\mu_i}$ are optimal solutions to \eqref{eq:sales-based_obj}.
\end{theorem}

\noindent
Intuitively, this result shows that formulation \eqref{eq:new_obj} serves as an asymptotically exact approximation to the original LP.
As the barrier parameter $\mu$ approaches zero, any accumulation point of the optimal solutions $\bar{y}_\mu$ is an optimal solution of the original problem.
This guarantees that the proposed algorithm asymptotically preserves optimality.
The proofs of Theorem~\ref{Theo_1_old} is provided in Online Appendix~\ref{Proof:Theo_1_old}.

We now establish a linear convergence rate of the solutions of \texttt{SPFOM} for the
variant \eqref{eq:new_obj} of SBLP following the idea of Theorem~3.1
of \citet{du2019linear}.

\begin{theorem}[Linear Convergence]\label{Theo_2_simple}
Let $\{y^t\}_{t \ge 0}$ be the sequence generated by Algorithm \ref{alg:sda_non} with $\mu \leq 1/(2R)$, where $R = \max\limits_{i \in [n]} \bigg\{ \sum\limits_{j \in [m]} w_{ij} r_j \big/ \sum\limits_{j \in [m]} w_{ij} \bigg\}$.
Then, there exist constants $q \in (0, 1)$ and $C > 0$ such that the iterates converge to the unique optimal solution $\bar{y}_{\mu}$ at a linear rate:
$$\mathbb{E} [\| y^t - \bar{y}_{\mu} \|^2] \le C \cdot q^t, \quad \forall t \ge 1.$$
\end{theorem}
\noindent
The proof of Theorem~\ref{Theo_2_simple} is provided in Online Appendix~\ref{Proof:Theo_2_simple}.
The theoretical parameterization, particularly the design of the values for $R$ and $\mu$, is intrinsically linked to the structural formulation of the SBLP. 
Specifically, $R$ represents the maximum weighted average revenue across all customers, serving as a global scale for expected returns within the system. 
By setting the step size $\mu = 1/(2R)$, the algorithm employs an adaptive learning rate that accounts for this revenue scale. 
This configuration ensures that updates to the primal variables $y^t$ are appropriately tempered relative to the objective function's sensitivity. 
Consequently, the $1/R$ scaling prevents numerical oscillations and satisfies the technical conditions required to achieve the linear convergence rate established in Theorem~\ref{Theo_2_simple}.

\begin{remark}
The proof of Theorem~\ref{Theo_2_simple} shows that $\mathbb{E}[\|y^{t+1}-\bar y_\mu \|]^2\leq \left( 1-\Theta(B/n) \right) \|y^t-\bar y_\mu \|^2$. This convergence result is different from those obtained for PDLP \citep{lu2021nearly} or cuPDLP (\citealt{lu2023cupdlp}). The linear convergence guarantees in PDLP and cuPDLP rely on the polyhedral sharpness of linear programs, which links distance to the solution set to global KKT residuals or duality gaps via Hoffman-type error bounds. In contrast, our SPFOM algorithm never forms such residuals: each update is computed from only $B$ randomly sampled customers out of a population of size $n$, and therefore operates in a fundamentally different information regime.
\end{remark}

\subsection{Parallel Computing}
The \texttt{SPFOM} (Algorithm \ref{alg:sda_non}) tackles a large-scale problem by sampling small-scale inner problems, solving them, and updating the solution iteratively until convergence is achieved. 
A crucial aspect of this method is the sampling process. 
To accelerate convergence for a huge-scale problem, one strategy is to simultaneously sample many small-scale inner problems and then solve and update them in parallel.
The computational complexity of each iteration of the parallelized \texttt{SPFOM} is $O(mB/k)$ when considering $k$ workers.

\begin{algorithm}[tbp]
\OneAndAHalfSpacedXI
\caption{Parallelized \texttt{SPFOM}}\label{alg:sda}
\textbf{Input:} Number of workers $k$, step size $\tau$, number of iterations $T$, batch size $B$, initial values $\{y_{ij}^0\}$, $\{\eta_j^{0}\}$\\
\textbf{For} $t=1,2,\ldots,T$: \\
\quad Randomly generate a subset of customers $\mathcal{I} \subset [n]$ with size $B\cdot k$\\
\quad \textbf{Parallel for} each worker $i \in [k]$ \textbf{do}:\\
\quad \quad Solve the inner optimization problem $z_i$ of size $B$ through golden-section search and \texttt{FAST-INNER}\\
\quad \textbf{End parallel}\\
\quad Update $y_{ij}^t$ for all $i \in \mathcal{I}$ and $j\in [m]$\\
\quad Execute the projected dual ascent update:\\
\quad \quad \textbf{For} \( j \in [m] \):\\
\quad \quad \quad \( \eta_j \leftarrow \max \left\{ 0, \eta_j - \tau \left( c_j - \sum\limits_{i \in [n]} y_{ij} \right) \right\} \)\\
\quad \quad \textbf{End for}\\
\textbf{End for}\\
\textbf{Output:} $\{y_{ij}^T\}$, $\{\eta_j^{T}\}$ and the objective value.
\end{algorithm}

\begin{corollary}\label{claim_1}
Consider Algorithm~3 with $k$ parallel workers.
Under the assumptions of Theorem~2, the expected distance to the smoothed optimal solution satisfies
{\small
\[
\mathbb{E}\bigl[\|y^{t+1}-\bar y_\mu\|^2\bigr]
\;\le\;
\bigl(1-\Theta(kB/n)\bigr)\,
\mathbb{E}\bigl[\|y^{t}-\bar y_\mu\|^2\bigr].
\]
}
\end{corollary}
\noindent
The proof of Corollary~\ref{claim_1}, which extends the analysis of Theorem~\ref{Theo_1}, is provided in Online Appendix~\ref{Proof:claim_1}.
The corollary shows that parallelization improves the contraction rate by effectively increasing
the number of sampled customers per iteration from $B$ to $kB$.
That is, while each worker processes a mini-batch of size $B$,
the aggregation across $k$ workers yields a linear speedup in the convergence rate.
This acceleration is orthogonal to the per-iteration computational complexity reduction,
which is discussed separately.

\subsection{Mapping the SBLP Solution to the CBLP}
A key challenge in applying the proposed framework is to map solutions from the SBLP back to interpretable decision strategies under the CBLP.
To this end, a recovery procedure is developed to reconstruct a feasible CBLP solution from a given SBLP output, thereby enabling actionable recommendations to be derived from the computationally efficient surrogate formulation.
\begin{lemma}\label{lemma_recover}
For each customer $i \in [n]$, sort the products by the ratio $y_{ij}/w_{ij}$ for each $j \in [m]$ in decreasing order.
Define the sets $\mathcal{S}_{i0}=\emptyset$ and $\mathcal{S}_{i\ell}=\{0, 1, .., \ell\}$ for each $\ell \in [n]$, where $\mathcal{S}_{i\ell}$ consists of the products ordered by the ration $y_{ij}/w_{ij}$ in decreasing order.
Hence, define
$$
    \alpha (\mathcal{S}_{i\ell}) = 
    \begin{cases}
        \left( \frac{y_{i \ell}}{w_{i \ell}} - \frac{y_{i \ell+1}}{w_{i \ell+1}} \right) 
        \sum\limits_{j \in \mathcal{S}_{i\ell}} w_{ij}
        \frac{1}{\lambda_i}  & \ell < n \\
         \frac{y_{i \ell}}{w_{i \ell}} 
         \sum\limits_{j \in \mathcal{S}_{i\ell}} w_{ij}
         \frac{1}{\lambda_i}  &  \ell = n.
    \end{cases}
$$
The value of $\alpha (\mathcal{S}_{i\ell})$ represents the probability that the set $\mathcal{S}_{i\ell}$ will be recommended.
\end{lemma}
\noindent
Here, the value of $\alpha$ corresponds to the value of $x$ in the CBLP.
The proof of this result follows the construction approach developed in \citet{topaloglu2012tractable} and \citet{gallego2015general}.

When the SBLP solution is mapped back to the CBLP via Proposition~\ref{lemma_recover}, the resulting assortment space is restricted to a subset of all possible product combinations.
For example, consider a setting with three products $\{1,2,3\}$, ordered according to the decreasing ratio $\frac{y_{ij}}{w_{ij}}$ as specified in Proposition~\ref{lemma_recover}.
Under the CBLP formulation, the recommendation strategy may assign probabilities over the full power set of products, yielding $2^3=8$ combinations (including the empty set): $\emptyset$, $\{1\}$, $\{2\}$, $\{3\}$, $\{1, 2\}$, $\{1, 3\}$, $\{2, 3\}$, $\{1, 2, 3\}$.
In contrast, the SBLP solution restricts attention to only a prefix-closed subset of assortments based on the sorted order, namely: $\emptyset$, $\{1\}$, $\{1, 2\}$, $\{1, 2, 3\}$. 
By limiting the feasible set to such structured combinations, the SBLP achieves a significant reduction in computational complexity, while still offering a close approximation to the CBLP's revenue performance.

\subsection{Numerical Study}
Numerical experiments with \texttt{SPFOM} are conducted to perform a comparative analysis with \texttt{PDLP}, which is a state-of-the-art first-order method for solving huge-scale linear programs~\citep{applegate2021practical, lu2023cupdlp}.
The computing environment for our experiments is as follows: 
CPU: AMD EPYC 7763 64-Core Processor, x86\_64, 2.50 GHz (max), 1.50 GHz (base), 128 threads (64 cores per socket);
RAM: 1.0 TiB total, 5.7 GiB used, 990 GiB free, 11 GiB buff/cache, 996 GiB available;
Operating System: Ubuntu 20.04.3 LTS, Linux kernel 5.4.0-94-generic;
Programming language: Julia 1.10.4.

\begin{figure}[tbp]
\centering
\caption{Illustration of convergence and optimality behavior of \texttt{SPFOM}: 
The left figure (a) shows the convergence of \texttt{SPFOM} with $m=100$ and $n=10,000$.
The right figure (b) compares the optimum values of \texttt{SPFOM} and \texttt{PDLP} for $m=100$ and $n \in [1000, \num{10000}]$. }
\subfigure[Convergence
]{\label{KModelIllus}\includegraphics[width=50mm]{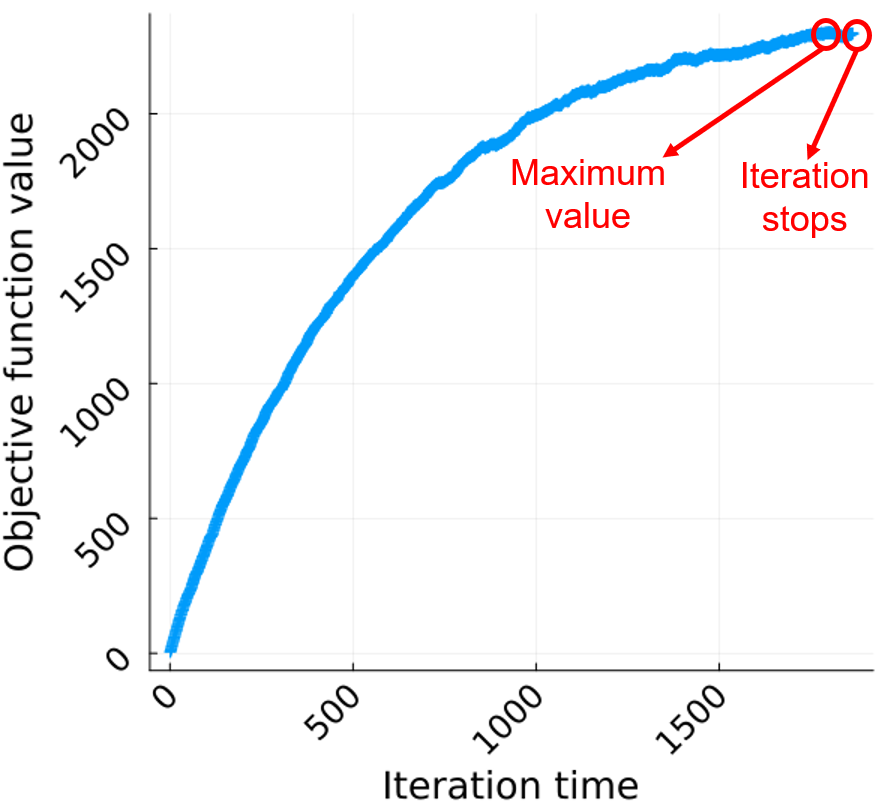}}
\hspace{5mm}
\subfigure[Optimal objective value]{\label{NModelIllus}\raisebox{-0.9mm}{%
      \includegraphics[width=70mm]{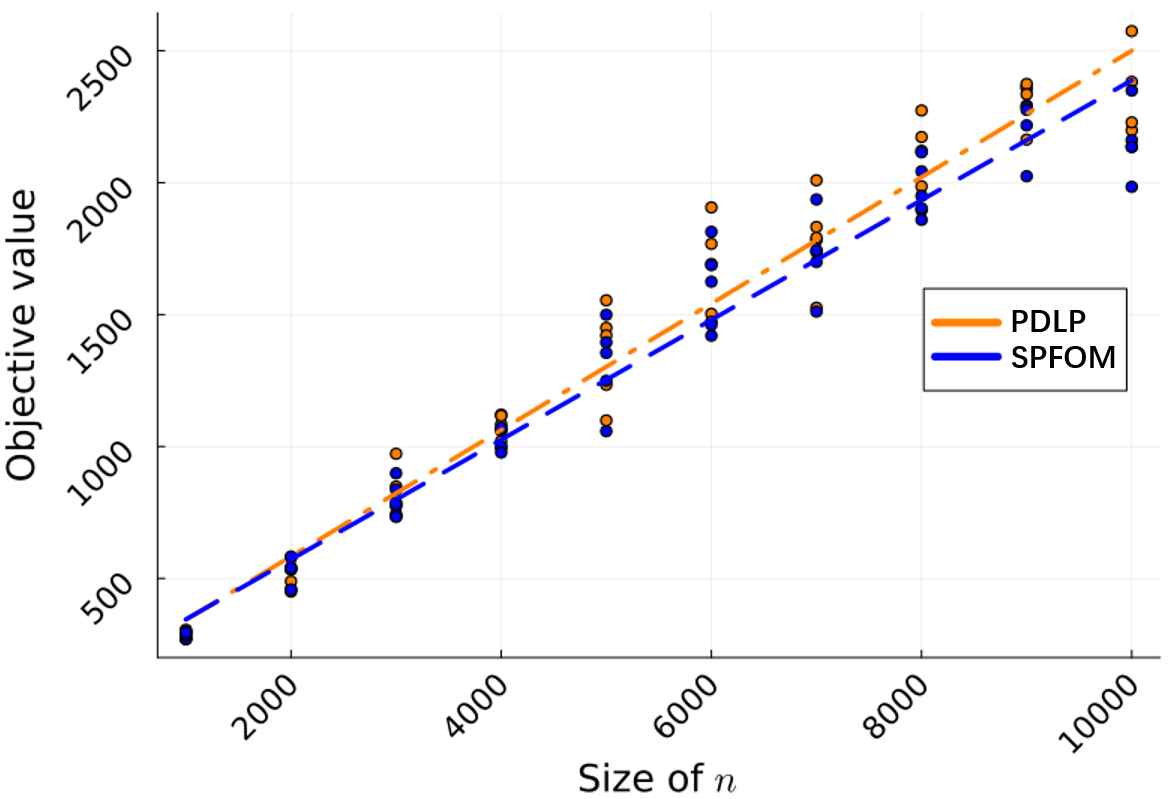}}}
\label{Fig:MDP_00}
\end{figure}

\paragraph{Convergence, Optimality, and Computational Time of \texttt{SPFOM}.}
Parameters $c$ (capacity), $r$ (price), $w$ (preference) are generated from a continuous uniform distribution $U(0, 1)$.
The other computational variables are set to standard values: $\tau = 0.1$, $\epsilon = 0.001$ and $\mu = 0.5$.
Batch size is set to $B=10$.
The convergence criterion is established by using a stagnation-based stopping criterion. 
The iteration process stops when the objective function value remains unchanged for a specified number of consecutive iterations, as shown in Figure~\ref{Fig:MDP_00}(a).
Based on empirical evidence, this number is set to $n/100$.
In practice, this number does not significantly affect the iterations, as we set a large value to ensure that the optimal or near-optimal solution is achieved.
Objective function values of \texttt{SPFOM} and \texttt{PDLP} are compared in Figure~\ref{Fig:MDP_00}(b). 
As $n$ grows, representing an increasing size imbalance between the two sides of the market, the objective value of \texttt{PDLP} becomes noticeably larger than that of \texttt{SPFOM}, whose performance remains near optimal.
Because \texttt{SPFOM} employs a primal-dual approach that embeds constraint~\eqref{eq:choice-based_1} into the objective function~\eqref{eq:choice-based_obj} (as reformulated in~\eqref{entro_obj}), its solutions may not fully satisfy constraint~\eqref{eq:choice-based_1} in finite samples.

\begin{figure}[tbp]
\centering
\caption{Computing speed of \texttt{SPFOM}: 
The figure (a) compares the computing speed of \texttt{SPFOM} and \texttt{PDLP} in a small-size market where $m=100$ and $n \in [1000, \num{10000}]$. 
The figure (b) shows the logarithmic computing time ($\log_{10}$) of both \texttt{SPFOM} and \texttt{PDLP}, averaged over several runs.
The figure (c) displays the computing time in hours for \texttt{SPFOM} in a large-sized market where $n \in [\num{10000}, \num{1000000}]$ and $m= \max \{100, n/1000\}$.
}\label{Fig:Speed}
\subfigure[Computing time (s)\color{red} \color{black}]{\label{speed_a}\includegraphics[width=49mm]{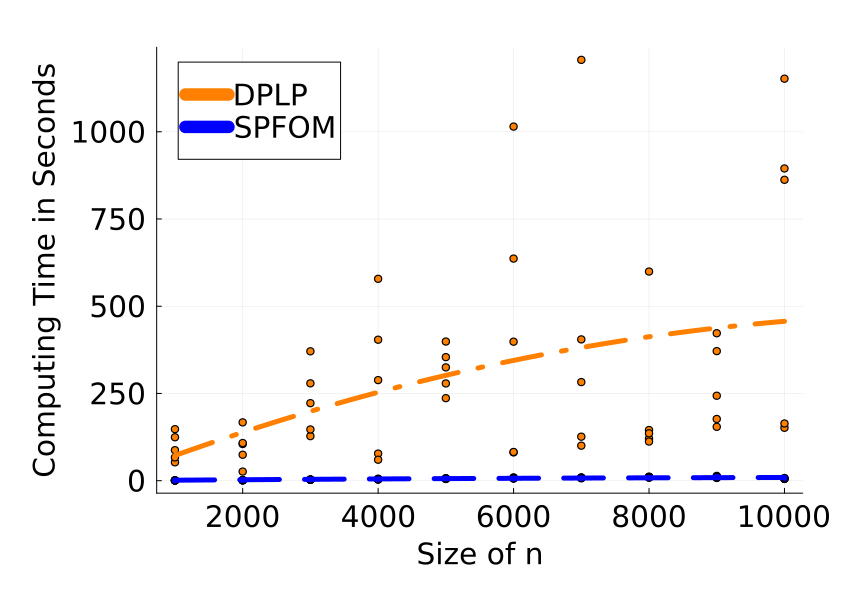}}
\hspace{2mm}
\subfigure[Logarithmic time ($\log_{10}$)]{\label{speed_b}\raisebox{0mm}{%
      \includegraphics[width=49mm]{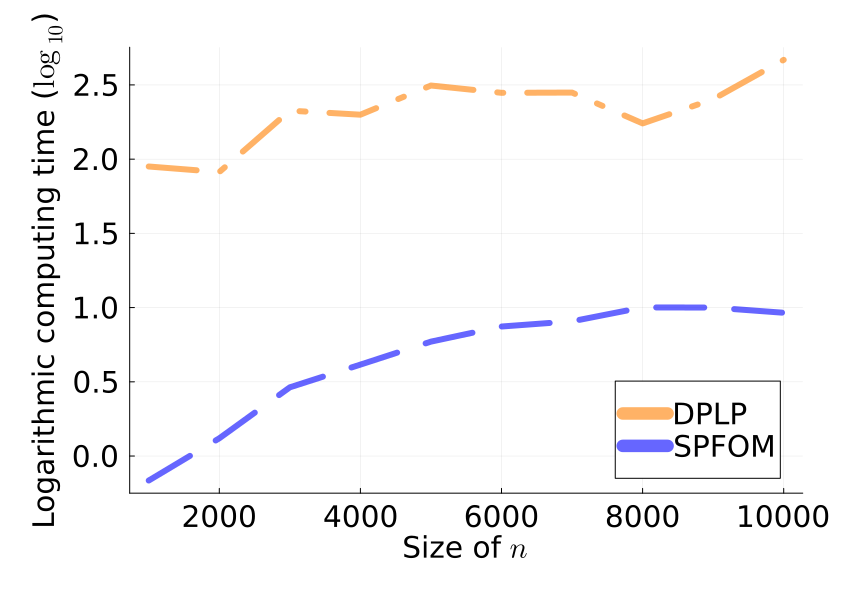}}}
\hspace{2mm}
\subfigure[Large-scale market]{\label{speed_c}\raisebox{-4.5mm}{%
      \includegraphics[width=57mm]{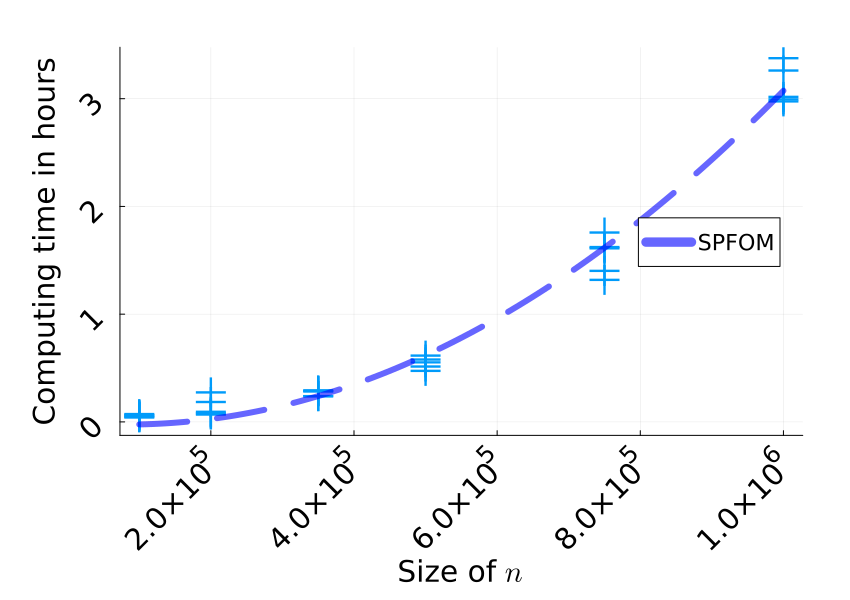}}}
\end{figure}

Computing speed of \texttt{SPFOM} compared to \texttt{PDLP} is shown in Figure~\ref{Fig:Speed}.
Figure~\ref{Fig:Speed}(a) shows individual computing examples where $n \in [\num{1000}, \num{10000}]$ for both methods. 
A polynomial regression highlights the increasing computing time as $n$ grows, with two key observations: (i) \texttt{SPFOM} computes significantly faster than \texttt{PDLP}, and (ii) \texttt{SPFOM} exhibits more stable computing times, whereas \texttt{PDLP} shows larger variation.
The higher variance in \texttt{PDLP}'s runtime arises from its adaptive restart and scaling mechanisms, which cause the number of iterations and memory access patterns to fluctuate across instances~\citep{applegate2021practical, lu2025cupdlp}.
Figure~\ref{Fig:Speed}(b) illustrates the averaged logarithmic computing time for both algorithms across multiple runs.
In Figure~\ref{Fig:Speed}(c), \texttt{SPFOM}'s performance is examined in a large-scale market with $n \in [\num{1000}, \num{1000000}]$ and $m = \max\{100, n/1000\}$. 
In this case, \texttt{SPFOM} computes recommendations for \num{1000} product types to \num{1000000} customers in around 5 hours on our CPU.
The performance of \texttt{PDLP} is not included in Figure~\ref{Fig:Speed}(c) since it crashes when $n$ exceeds \num{100000} due to out-of-memory issues.
This happens because \texttt{PDLP} significantly increases storage demands during the model pre-solve step as problem size grows, causing its computation time to rise dramatically in our experiments.

\begin{table}[tbp]
\centering
\caption{
Average computation time of the parallelized \texttt{SPFOM} algorithm for $m=100$, $n=\num{100000}$ and $k \cdot B=\num{3000}$ with different numbers of workers (core numbers) $k \in [1, 2, 4, 6, 8]$ and optimality gap larger than $95 \%$. 
}
\renewcommand*{\arraystretch}{1.2}
\begin{tabular}{c|c|c|c}
\toprule
  $k$  & Computing time (s) & Logarithmic time ($\log_{10}$) & Iteration times \\
\midrule
Single core &    1733.5    &   3.2  &  1014  \\
2 cores     &    1242.7    &   3.1  &  636   \\
4 cores     &    101.5     &   2.0  &  47    \\
6 cores     &     36.7     &   1.6  &  15    \\
8 cores     &    24.9      &   1.4  &  9     \\
\bottomrule
\end{tabular}
\label{Tab:Parallel}
\end{table}

\paragraph{Speedup of Parallel Computing.}
In large-scale optimization, increasing problem size leads to significantly higher computational costs. 
For example, as the problem grows, the column generation process becomes more complex and resource-intensive. 
The storage and computational power required by the algorithm rise non-linearly with the number of decision variables as it searches for the optimal point. 
Additionally, solvers perform a \textit{pre-solve} step before finding the solution, but for large problems, this step alone can overwhelm the computer's storage capacity.

The \texttt{SPFOM} algorithm does not pre-possess many decision variables at each iteration step.
Instead, it computes and updates the solution for many small-scale problems, thereby avoiding the aforementioned computational limitations.
As shown in Table~\ref{Tab:Parallel}, the computation time decreases significantly with more parallel workers. 
In steps 3-5 of the parallelized \texttt{SPFOM} algorithm~\ref{alg:sda}, small-scale inner problems $z_i$ are computed in parallel. 
Multiple inner problems are sampled simultaneously, and then $\eta_j$ is updated using all the computed results. 
This effectively increases the iteration step size, reducing the number of steps needed for \texttt{SPFOM} to converge, thus speeding up the overall computation.

Compared to the non-parallelized \texttt{SPFOM} (Algorithm~\ref{alg:sda_non}), the parallelized \texttt{SPFOM} (Algorithm~\ref{alg:sda}) includes an additional data transfer step. 
Specifically, when inner problems are assigned to different workers, their parameters and results need to be transferred, a process referred to as \textit{fetch} in \texttt{Julia 1.10.4}.
This introduces two effects in our experiments. 
First, parallel \texttt{SPFOM} is more efficient than the non-parallel version when the batch size $B$ is large, as fetch time is relatively short compared to computation time in each iteration. 
Second, when the number of workers is large and iterations are few, the marginal speedup decreases because a larger proportion of time is spent on fetching.

\section{Extensions}\label{sec5:multiNRM}
\subsection{Extension 1: Multi-Period Setting} 

The proposed framework can be extended to a multi-period setting to better capture dynamic environments in which both demand and inventory evolve over time.
This section illustrates how the \texttt{SPFOM} algorithm can be incorporated into a bid-price control policy for such dynamic settings, and evaluates its performance through numerical experiments.

\subsubsection{Multi-Period LP Formulation}
We present the multi-period CBLP formulation as follows. This formulation contains dynamic assortment decisions for multiple stages, indexed by $t$, from $1$ to $T$, where the customer choice distribution is i.i.d..
{
\begin{align}
\max_{\{x_i^t(S)\}} \quad & \sum_{t=1}^T \sum_{i \in [n]} \sum_{S \in \mathcal{S}} x_i^t(S) \cdot \sum_{j \in S} \pi_{ij}(S)  \cdot r_j \label{eq:cblp-obj} \\
\text{s.t.} \quad 
& \sum_{t=1}^T \sum_{i \in [n]} \sum_{S \in \mathcal{S}} x_i^t(S) \cdot \textbf{1}(j \in S) \cdot \pi_{ij}(S) \leq c_j^1, && \forall j \in [m] \label{eq:cblp-inv} \\
& \sum_{S \in \mathcal{S}} x_i^t(S) = \lambda_i^t, && \forall i \in [n],~ t \in [T] \label{eq:cblp-rate} \\
& x_i^t(S) \geq 0, && \forall i \in [n],~ S \in \mathcal{S},~ t \in [T]. \label{eq:cblp-nonneg}
\end{align}
}
The corresponding multi-period SBLP is given as follows.
{
\begin{align}
\max_{\{y_{ij}^t\}} \quad
& \sum_{t = 1}^T \sum_{i \in [n]} \sum_{j \in [m]} y_{ij}^t \cdot r_j \label{eq:sblp-obj} \\
\text{s.t.} \quad
& \sum_{t = 1}^T \sum_{i \in [n]} y_{ij}^t \leq c_j^1, && \forall j \in [m] \label{eq:sblp-inv} \\
& \sum_{j = 0}^m y_{ij}^t = \lambda_i^t, && \forall i \in [n],~ t \in [T] \label{eq:sblp-balance} \\
& \frac{y_{ij}^t}{w_{ij}} \leq \frac{y_{i0}^t}{w_{i0}}, && \forall i \in [n],~ j \in [m],~ t \in [T] \label{eq:sblp-ratio} \\
& y_{ij}^t \geq 0, && \forall i \in [n],~ j \in \{0\} \cup [m],~ t \in [T] \label{eq:sblp-nonneg}
\end{align}
}
\noindent
Compared with the single-period formulation, the time index $t$ is introduced to capture temporal dynamics.
The total inventory $c_j^1$ is given initially and gradually depletes across periods according to the realized recommendations and customer choices. 
Consequently, both the arrival intensities $\lambda_i^t$ and preference weights $w_{ij}^t$ evolve over time.
The objective is to maximize the cumulative expected revenue from a sequence of product recommendations over the planning horizon.

\subsubsection{Bid-Price Control}
Bid price control is a widely used policy in multi-period RM problems for efficiently allocating limited resources across incoming customer requests \citep{talluri1998analysis}. 
The core idea is to assign an implicit value, known as a \emph{bid price}, to each unit of resource, and to accept a customer request only if the associated revenue exceeds the total value of the resources consumed. 
Formally, let $v_j^t$ denote the bid price of resource $j$ at period $t$, which represents the opportunity cost of consuming one additional unit of resource $j$. 
Under the bid price control rule, product $j$ is recommended to customer $i$ at time $t$ only if the marginal profit satisfies $r_j^t - v_j^t \ge 0$, where $r_j^t$ is the price of product $j$. 
In our framework, $v_j^t$ corresponds to the dual variable $\eta_j^t$ obtained from the \texttt{SPFOM} formulation, which captures the shadow value of the remaining capacity. 
Intuitively, this rule ensures that resources are allocated only when the immediate gain $r_j^t$ exceeds the future value of capacity $v_j^t$, thereby balancing short-term revenue and long-term availability.

Although the multi-period setting involves sequential decision-making over time, our formulation is fundamentally different from standard online algorithms.
In online learning frameworks, the decision-maker observes feedback (e.g., realized demand) and adaptively updates decisions without knowledge of future arrivals, typically aiming to minimize \textit{regret} relative to an offline benchmark.
In contrast, our multi-period model assumes that the stochastic demand and preference distributions are known in advance, allowing the problem to be formulated as a deterministic, expectation-based linear program.
Accordingly, our focus lies in developing an efficient primal-dual optimization algorithm for solving large-scale structured LPs, rather than in minimizing regret under partial information.
Among possible control frameworks, bid-price control offers a tractable and interpretable approach that aligns naturally with LP-based formulations~\citep{talluri1998analysis}.
Our work therefore enhances this classical framework by improving its computational scalability through the proposed \texttt{SPFOM} algorithm.

\subsubsection{Global Market vs. Market Segmentation Decomposition}

Under the bid-price control framework, two decision policies are compared: a global optimization (GO) policy based on \texttt{SPFOM}, and a segment-level policy derived from market segmentation decomposition (MSD).

\paragraph{GO Policy.}
For each batch, products with zero inventory are first removed.
If the number of available products does not exceed the recommendation limit, all remaining products are recommended; 
otherwise, the corresponding SBLP is solved using \texttt{SPFOM} to obtain global dual prices.
A bid-price control rule is then applied, recommending products whose marginal revenue exceeds their bid prices.
Customer purchases are simulated under the BAM framework, and inventory levels are updated accordingly.

\paragraph{MSD Policy.}
The MSD policy processes each market segment independently within each batch.
For each segment, if the number of available products is within the recommendation limit, all are recommended; 
otherwise, the segment-specific CBLP is solved to obtain local dual prices, which are used to guide bid-price control rule's recommendations.
Purchases are simulated separately for each segment, and the resulting outcomes are aggregated to update overall inventory and revenue.

\begin{figure}[tbp]
\centering
\caption{
We simulate a market with 100 batches of 1000 customers each, where 2 products are recommended per batch.
The platform offers 40 products whose prices are drawn uniformly from $[1, 20]$ and initial inventories from $[1, 2000]$.
Each customer's preference toward each product is independently drawn from $[0, 1]$.
In every batch, the MSD policy partitions customers into 10 segments, whereas the global policy operates on the batch as a whole.
All results reported in the figure are averaged over 30 independent simulation runs.
} 
\subfigure[Cumulative Revenue]{\label{speed_a}\includegraphics[width=53mm]{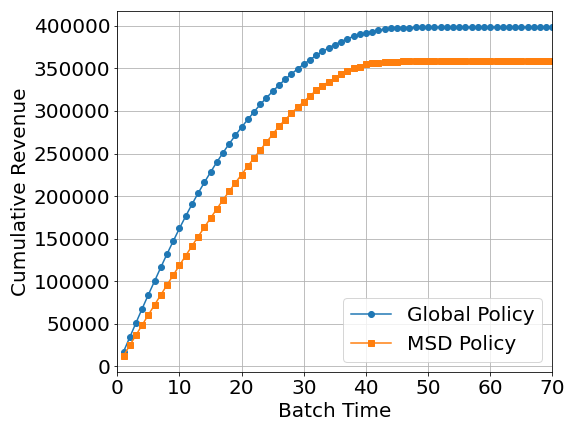}}
\hspace{0mm}
\subfigure[Revenue per Batch]{\label{speed_b}\raisebox{0mm}{%
      \includegraphics[width=53mm]{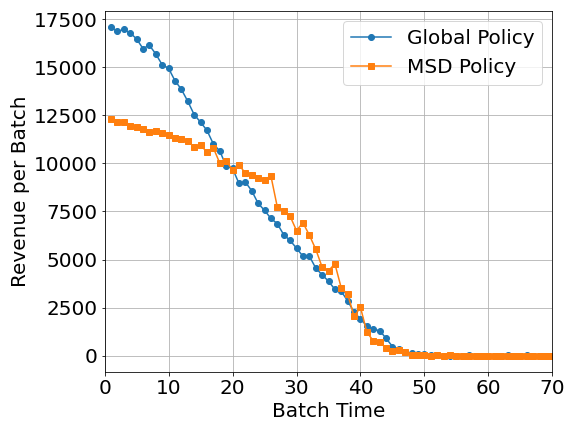}}}
\hspace{0mm}
\subfigure[Quantity Sold per Batch]{\label{speed_c}\raisebox{0mm}{%
      \includegraphics[width=53mm]{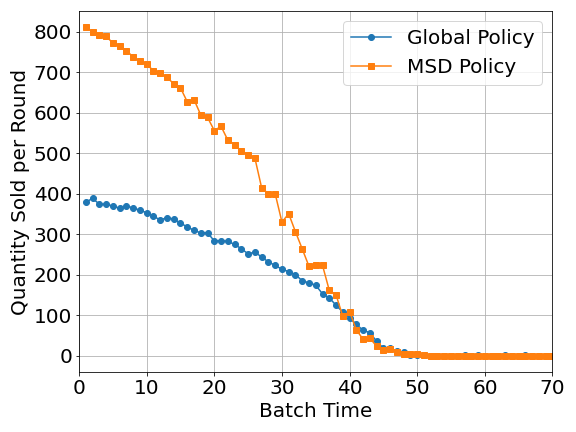}}}
\label{Fig:Speed222}
\end{figure}

\begin{result}\label{result_1}
When the overall inventory is sufficiently abundant, the GO policy identifies and recommends higher-revenue products while utilizing inventory more efficiently than the MSD policy.
\end{result}

Figure~\ref{Fig:Speed222} compares the GO and MSD policies across key performance metrics.
As shown in Figure~\ref{Fig:Speed222}(a), the GO policy consistently achieves higher cumulative revenue than the MSD policy under identical inventory and demand conditions.
This difference reflects the distinct recommendation mechanisms of the two policies (Result~\ref{result_1}): by leveraging global information through \texttt{SPFOM}-computed shadow prices, the GO policy constructs higher-value assortments.
Figures~\ref{Fig:Speed222}(b) and~(c) further illustrate the temporal dynamics.
In early batches, the GO policy achieves higher revenue per batch while selling fewer units, whereas the MSD policy sells more units but with lower per-batch revenue.
As a result, the MSD policy rapidly depletes inventory and suffers a sharp revenue collapse after around batch 30, while the GO policy maintains a more gradual decline.
These results highlight the advantage of the GO policy in sustaining profitability and managing inventory consumption over time, particularly under medium- to long-term operational horizons.
From an economic perspective, this advantage stems from its ability to capture the long-term marginal value of inventory through global shadow prices and to use such signals to coordinate assortment decisions across products.
A similar economic principle underlies recent work in multi-product inventory systems, where shadow values are used to guide allocation and upgrading decisions in a forward-looking manner~\citep{tang2025multiproduct}.
Recent studies have also shown that even under complex sequential choice and search dynamics, optimal or near-optimal assortment decisions can be guided by simple reduced-form signals, such as last-choice probabilities, which play a role analogous to shadow values in dynamic systems~\citep{gao2023assortment}.

\subsection{Extension 2: Network Revenue Management with Bundling}

We extend our framework to Network Revenue Management (NRM) with bundling. 
In this setting, $n$ customers (indexed by $i$) choose from $m$ bundles (indexed by $j$) composed of $L$ base products (indexed by $\ell$). 
The resource constraint \eqref{eq:sales-based_1} is generalized as:
\begin{equation}\label{eqa:new_NRM_1}
    \sum_{i\in [n]} \sum_{j\in [m]} \tilde{B}_{\ell j} y_{ij} \le c_\ell, \quad \forall \ell \in [L],
\end{equation}
where $y_{ij}$ represents the selection of bundle $j$ by customer $i$. 
This formulation introduces a binary resource-request incidence matrix $\tilde{B} \in \{0, 1\}^{L \times m}$, where each element $\tilde{B}_{\ell j}$ indicates whether bundle $j$ consumes base product $\ell$, thereby capturing the joint consumption of resources.

\begin{corollary}\label{coro_new_2}
    Suppose $\sum\limits_{j \in [m]} ( r_j - \eta_j^t ) \leq \tilde{R}$ for all $t$.
    Then, for a $\mu$ satisfying $ \mu \leq \frac{\rho_{\max}^2}{2\tilde{R}} $, Algorithm~\ref{alg:sda_non} converges linearly to the unique optimal solution.
    Here, $\rho_{\max}$ denotes the maximum singular value (spectral norm) of the constraint matrix $\tilde{A}$ associated with the general formulation of \eqref{eq:sales-based_obj}, \eqref{eqa:new_NRM_1}, and \eqref{eq:sales-based_2}-\eqref{eq:sales-based_4}. 
\end{corollary}

Corollary 2 directly extends the results of Theorem~\ref{Theo_1}, the full details and proof of which are deferred to the Online Appendix.
Intuitively, incorporating the resource-request matrix $\tilde{B}$ generalizes the structure of the standard constraint matrix $A$ while preserving its full row rank property. 
By characterizing the spectral properties of this augmented matrix, we adapt the conditions of Theorem~\ref{Theo_1} to verify the linear convergence of Algorithm~\ref{alg:sda_non} in the bundling context.
These findings reaffirm the robustness and efficiency of the algorithm even under network-based resource-sharing constraints.

\section{Computational Case Study for ZOZOTOWN}\label{sec6:case}
We evaluate the proposed method in a large-scale empirical setting using real-world data from ZOZOTOWN. 
In particular, we compare two approaches: global recommendation optimization and segmented recommendation optimization.
Further details regarding the experimental setup and specific data preprocessing are provided in Online Appendix~\ref{App_2}.

\subsection{Empirical Context and Data Overview}
Customers browsing a product on ZOZOTOWN are presented with a panel of recommended items.
We define a customer’s action of clicking and subsequently adding a recommended item to the cart as a realized choice (treated as a ``purchase'' record for simplicity), a convention that is consistent with recent work incorporating click behavior into choice models to improve the prediction of customer decisions \citep{aouad2025click}.
Only items with available inventory are displayed, ensuring that all observed choices are feasible.
Our analysis uses large-scale proprietary transaction data from a single day in June 2024.
The dataset contains the 1000 most popular SKUs and records \num{1372671} realized choices made by \num{38746} unique customers.
For our sequential optimization framework, the 24-hour period is discretized into 96 intervals of 15 minutes each, which defines our operational horizon.
This temporal segmentation is purely methodological, does not reflect any inherent customer or product characteristics, and also supports standard anonymization of customer and product information.

\subsection{Estimation of Preference Weights in the Choice Model}
In this paper, we estimate customer–SKU preference weights using maximum likelihood estimation (MLE). 
We adopt the standard multinomial logit (MNL) specification for its tractability in large-scale settings \citep{berbeglia2022comparative}.
To capture preference heterogeneity, we first partition users into $K$ behavioral segments using $k$-means clustering on standardized customer features, where $K$ is determined through an iterative refinement process to ensure sufficient sample sizes for reliable estimation.
We then estimate a separate MNL model for each segment $C_\kappa$ (where $\kappa = 1, \dots, K$). 
The preference weight $w_{ij}$ for customer $i$ belonging to segment $C_\kappa$ is modeled as:
\[
w_{ij} = \alpha_\kappa + x_j^{\top} \boldsymbol{\theta}_\kappa + \varepsilon_{ij} \quad \text{for each} \quad i \in C_{\kappa},
\]
where $\alpha_\kappa$ and $\boldsymbol{\theta}_\kappa$ denote the intercept and marginal utility coefficients specific to segment $\kappa$.
The feature vector $x_j$ consists of the product's scaled price and one-hot encoded color attributes. 
Assuming i.i.d. Gumbel shocks, the MNL probabilities are computed using a sampled choice set that includes the observed chosen product and several randomly selected non-chosen alternatives.
While more flexible structures (e.g., mixed logit) exist, they are computationally infeasible at our dataset’s scale.
The segmented MNL model achieves a practical balance between behavioral interpretability, estimation stability, and computational tractability, providing consistent preference estimates for evaluating our global optimization framework, \texttt{SPFOM}.

\subsection{Large-Scale Market Optimization via \texttt{SPFOM}}

In this section, we compare the platform’s expected revenue and its inventory overload ratio under the two recommendation approaches over a fixed multi-period horizon. 
The results in Table~\ref{Tab:Case_1} represent the average performance across 96 time periods (15-minute operational batches).
{\small
\begin{table}[tbp]
\centering
\caption{
Comparison of performances of global optimization and MSD approaches.
}
\renewcommand*{\arraystretch}{1.2}
\begin{tabular}{c|ccccc}
\toprule
\raisebox{0mm}{Metric} &  MSD($\#S=5$)   &  MSD($10$)  & MSD($25$)  &  MSD($50$) & \textbf{GO} \\ \midrule
Total expected revenue & $4.2 \mathrm{e}7$  & $3.3 \mathrm{e}7$ & $2.4 \mathrm{e}7$ & $1.7\mathrm{e}7$ & $\textbf{7.6e7}$  \\ 
Inventory overload ratio  & $0.83\%$  & $1.25\%$ & $1.87\%$ &  $2.18\%$  & $\textbf{0.39\%}$   \\
\bottomrule
\end{tabular}
\label{Tab:Case_1}
\end{table}
}
The table shows the superior performance of the centralized GO approach, which achieves the highest total expected revenue ($7.6\mathrm{e}7$) while maintaining the lowest inventory overload ratio ($0.39\%$). 
Conversely, the MSD performance degrades significantly as the number of segments ($\#S$) increases (revenue drops from $4.2\mathrm{e}7$ at $\#S=5$ to $1.7\mathrm{e}7$ at $\#S=50$).
This outcome confirms that decentralization introduces inherent inefficiency under resource constraints. 
The GO strategy, with its unified global view, enforces the shared inventory constraint optimally, resulting in higher revenue and a reduced risk of stockouts. 
MSD, however, suffers from fragmented resource allocation and the accumulation of local estimation errors, causing a catastrophic loss in revenue potential and a severe violation of the global inventory constraint as market granularity increases.

\subsection{Multi-Period Optimization via Bid-Price Control}
This section compares the time-averaged results of the centralized GO strategy against the decentralized MSD strategy to determine their respective capacities for balancing immediate revenue against long-term inventory preservation.
\begin{figure}[tbp]
\centering
\caption{Comparison of GO and MSD across 96 multi-period batches.
} 
\subfigure[Cumulative revenue]{\label{speed_a}\includegraphics[width=60mm]{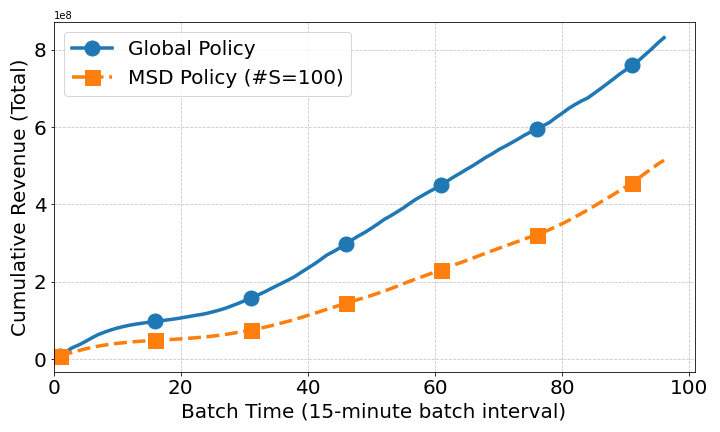}}
\hspace{0mm}
\subfigure[Revenue per batch]{\label{speed_b}\raisebox{0mm}{%
      \includegraphics[width=60mm]{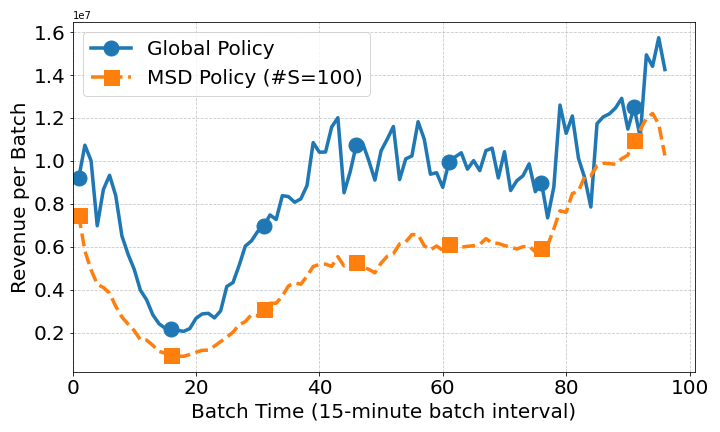}}}
\hspace{0mm}
\subfigure[Remaining inventory]{\label{speed_c}\raisebox{0mm}{%
      \includegraphics[width=60mm]{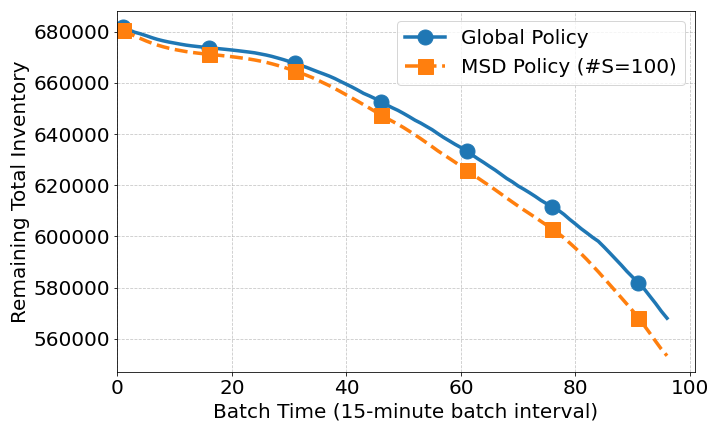}}}
\hspace{0mm}
\subfigure[Quantity sold per batch]{\label{speed_c}\raisebox{0mm}{%
      \includegraphics[width=60mm]{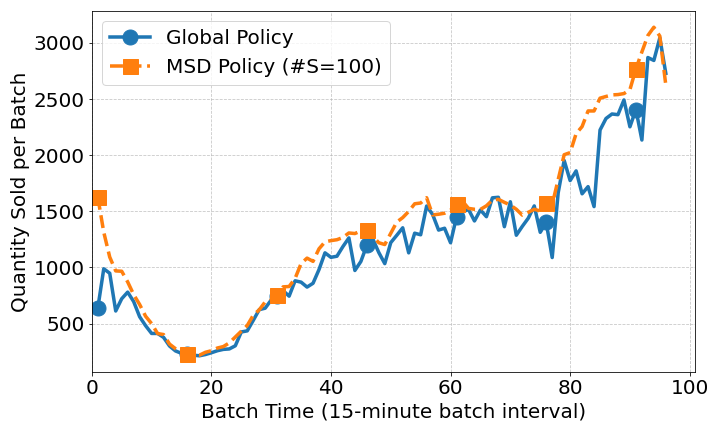}}}
\label{Fig:RealMulti}
\end{figure}
As shown in Figure~\ref{Fig:RealMulti}, the Global Policy (GO) significantly outperforms the MSD Policy in generating revenue and fulfilling demand. 
Specifically, Figure~\ref{Fig:RealMulti}(a) and (b) reveal substantially higher cumulative revenue and revenue per batch for GO. 
Figure~\ref{Fig:RealMulti}(d) shows that GO achieves a lower quantity sold per batch, leading to a slower depletion of inventory, as illustrated in Figure~\ref{Fig:RealMulti}(c).
The GO policy outperforms MSD by internalizing global inventory scarcity through a unified shadow price, avoiding over-allocation to low-value transactions. 
In contrast, MSD’s decentralized structure assigns effectively zero scarcity cost, leading to systematic misallocation and stockouts. 
Two diagnostic metrics in Online Appendix~\ref{App_secB4}, product sales diversity (entropy) and mean opportunity cost, further highlight these allocation differences through the distribution of chosen products.

\section{Conclusions and Future Research}\label{sec7:conclusion}

We propose \texttt{SPFOM}, a novel first-order algorithm designed to efficiently solve large-scale RM problems with combinatorial assortment constraints.
The method achieves theoretical guarantees while exhibiting competitive empirical performance, scaling effectively to high-dimensional instances where conventional solvers fail to remain practical.
In addition to its computational scalability, \texttt{SPFOM} facilitates fine-grained recommendation control through dual-based bid price policies.
Globally coordinated decision-making enabled by \texttt{SPFOM} yields superior outcomes compared to decentralized, segment-based approaches, achieving higher revenue and more efficient inventory utilization.
The proposed framework can also be naturally extended to NRM problems, where decisions are coupled through shared resource capacities.
Future research directions include incorporating data-driven learning modules to jointly estimate customer preferences and optimize assortments in an adaptive fashion~\citep{ferreira2018online, jia2021multi, cheng2025social}. 
Extending the framework to online or dynamic settings would better capture the operational realities of real-time systems.
Another promising avenue is the integration of inventory control policies into the decision-making process~\citep{liang2021assortment}, thereby enhancing coordination between demand learning, assortment optimization, and stock management, particularly under supply constraints or uncertain replenishment dynamics.

\bibliographystyle{informs2014}
\bibliography{mybib.bib}

\ECSwitch
\OneAndAHalfSpacedXI
\ECHead{Online Appendices}
\section{Omitted Proofs}\label{App_1}

\subsection{Proof of Lemma~\ref{lemma_1}}\label{Proof:Lemma1}
Without loss of generality, we assume the indices $j$s are ordered such that $r_1-\eta_1\geq r_2-\eta_2\geq \cdots \geq r_m-\eta_m$. Then, note the optimization problem for solving $z_i$ can be solved equivalently through a dynamic programming recursion, as follows: (let $s_1=\lambda_i$)
{\small
\begin{align*}
    \left\{\begin{array}{ll}
        V_{m+1}(0):= 0 &  \\
        V_j(s_j):= & \max\limits_{y_{ij}:0\leq y_{ij}\leq \min\{\frac{w_{ij}}{w_{i0}}y_{i0},s_j\}} (r_j-\eta_j)y_{ij} \nonumber \\
        &+V_{j+1}(\underbrace{s_j-y_{ij}}_{:=s_{j+1}}), \forall~j=1,\ldots,m
    \end{array} \right..
\end{align*}
}
By an induction argument, we can show that $V_j'(\cdot)$ is bounded above by $r_j-\eta_j$. Then, we derive the greedy approach (try to set $y_{ij}$ as large as possible) must be optimal due to the ordering of $r_j-\eta_j$.
\hfill$\blacksquare$

\subsection{Proof of Theorem~\ref{Theo_1_old}}\label{Proof:Theo_1_old}
We follow the notation used in this paper. To apply the quadratic penalty theory from \citet{NocW06a}, we first define a base feasible set $\mathcal{\bar{Y}}$ incorporating constraints (7), (8), and (9):
\begin{align*}
\mathcal{\bar{Y}} = \Bigg\{ y_{ij} \mid \sum_{j=0}^m y_{ij} = \lambda_i, \frac{y_{ij}}{w_{ij}} \le \frac{y_{i0}}{w_{i0}}, &y_{ij} \ge 0, \\
&\forall i\in [n], j\in [m] \Bigg\}.
\end{align*}
Since $\lambda_i$ and $w_{ij}$ are fixed, finite, and positive parameters, $\mathcal{\bar{Y}}$ is a closed and bounded set, hence it is compact.
To address the resource inequality constraints \eqref{eq:sales-based_1}, we introduce non-negative slack variables $s_j \ge 0$ to reformulate them as equalities: $c_j - \sum_{i \in [n]} y_{ij} - s_j = 0$.
Following the construction in Theorem 17.1 of \citet{NocW06a}, the augmented objective function $Q$ for our problem is defined as:
\begin{align}
\min_{\substack{\eta_j \\ j\in [m]}} 
    \max_{\substack{y_{ij} \in \mathcal{\bar{Y}} \\ i\in [n] \\ j\in \{0\}\cup[m]}} 
&\sum_{i\in [n]}\sum_{j\in [m]} y_{ij} r_j \nonumber  \\ 
&+ \sum_{j\in [m]}\eta_j \left(c_j - \sum_{i\in [n]} y_{ij} - s_j\right) + \frac{\mu}{2}\sum_{j\in [m]}\eta_j^2,  \label{eq:proof_saddle}
\end{align}
where $\frac{\mu}{2}\sum\limits_{j\in [m]}\eta_j^2$ is the quadratic regularization term. 
To eliminate $\eta$, we find the value $\eta_j^*$ that minimizes this expression. 
Taking the partial derivative of \eqref{eq:proof_saddle} with respect to $\eta_j$ and setting it to zero:
\begin{align*}
\frac{\partial \eqref{eq:proof_saddle}}{\partial \eta_j} = c_j - \sum_{i \in [n]} y_{ij} - s_j + \mu \eta_j = 0 \quad \text{for each } j \in [m].
\end{align*}
Solving for the optimal dual variable, we get $\eta^*_j = \left(\sum\limits_{i \in [n]} y_{ij} + s_j - c_j \right) / \mu$. 
Substituting this back into the objective function, we obtain the regularized objective \eqref{eq:new_obj}.
We note that as $\mu \rightarrow 0$, the effective penalty coefficient $1/\mu$ tends to infinity, which is equivalent to the case in the original Theorem 17.1 of \citet{NocW06a} where the penalty parameter goes to infinity. 
Since $\mathcal{\bar{Y}}$ is compact, any sequence of solutions $\{\bar{y}_{\mu}\}$ has at least one limit point. 
According to Theorem 17.1, every such limit point must satisfy the constraints and be an optimal solution to the original problem (6).
\hfill$\blacksquare$

\subsection{Preliminaries on Linear Convergence}

We introduce some notations.
By rewriting in matrix form, the problem~\eqref{entro_obj}-\eqref{entro_cons_4} considered is equivalent to
\[
	\min_{y \in \bY} \max_{\eta \geq 0}\, -\left( r^{\top} y +
	\eta^{\top} (c - Ay) + \frac{\mu}{2}\eta^{\top} \eta\right),
\]
where $y := (y_1,y_2,\dotsc,y_n)$, $y_i := (y_{i0},y_{i1},\dotsc,y_{im})$, $r \in \R^{m(n+1)}$ is the vector concatenating $(0,r_1,r_2,\dotsc,r_m)$ for $n$ times, $\bY := \bY_1\times\cdots\times \bY_n$, where each $\bY_i$ denotes the feasible set of the constraints eqs.~\eqref{entro_cons_1} to \eqref{entro_cons_3} for $y_i$, and $A \in \R^{m \times m(n+1)}$ is formed such that each row is split into $m$ blocks of $n+1$ entries, and the $i$-th row is formed by repeating $(0,e_i)$ for $m$ times, where $e_i$ is the $i$-th standard unit
vector of $\R^{n}$.

\begin{theorem}\label{Theo_1}
Consider Algorithm~\ref{alg:sda_non}
for some $\mu \in (0,1)$. 
Let $\rho_{\min}$ and $\rho_{\max}$ respectively be the smallest and the
largest singular values of $A$.
Assume that $A$ has full row rank, there is a constant $C \in (0, 2
\mu B / (n (\rho_{\min}^2 + \rho_{\max}^2)))$ such
that $\lambda_i / |\sum_{j} r_j - \eta^t_j| < C$
for any $i$ for all $t$.
If $C, \mu, \tau$ and $B$ satisfy
{\small
\begin{align*}
	& \tau \leq \mu^{-1}, \quad \sqrt{1 - \frac{B C
		\rho_{\min}^2}{n\mu}} < 1, \quad C \rho_{\max} < 1, \\
        & C
		\frac{\rho_{\max}^3}{\mu^2} < 1,  
        \quad \sqrt{1 - \frac{B \tau
		\mu}{n}} + C \frac{\rho_{\max}^2}{\mu} < 1
\end{align*}
}
then the expected distance from the iterates $\{y^t\}$ to $\bar
y_{\mu}$ converges at a linear rate. More specifically,
let $\lambda := \max_{i \in [n]} \lambda_i$,
there exists $c>0$ such that $0\leq q \leq 1-c\cdot(B/n)$ such that
{\small
\begin{align*}
\mathbb{E}[\|y^t - \bar{y}_{\mu}\|]^2 \leq & q^t \bigg(
3 \|\eta^0\|^2 +
\frac{3}{\mu}\|c\|^2 + \frac{3}{\mu} m^2 (n+1) \|y_0\|^2  \\
&+ \lambda^2 + 3 \mu^2 \|\eta^0\|^2 + 3\|c\|^2 + m(n+1) \lambda \bigg).
\end{align*}
}
\end{theorem}

\textit{Proof of Theorem~\ref{Theo_1}.}
We first reformulate the problem via the conjugate function $g^*$.
The primal objective is $\mu$-strongly convex and $\mu$-smooth, implying that $g^* (Ay)$ is $\sigma$-strongly convex and $L$-smooth with respect to $y$.
We then identify a suitable step size $\theta$ for the block coordinate gradient descent update.
Under this step size, we construct a ghost algorithm and, using existing results on randomized block coordinate descent, show that the primal error $\Ex \|y^{t} - \bar y_{\mu}\| $ contracts linearly in expectation.
Next, we analyze the dual update and derive a linear contraction bound for the dual residual $\Ex\|\eta^{t} - \nabla g^*(Ay^{t})\|$ by combining the primal contraction with the Lipschitz continuity of the dual gradient.
Finally, the coupled primal–dual recursions yield the claimed linear convergence rate.

We define
\[
	f(y) := -r^{\top} y,  \quad g(\eta)
	:= \frac{\mu}{2} \eta^{\top} \eta + c^{\top} \eta,
\]
then we can easily see that $\|A\| \leq m\sqrt{n+1}$, and $g$ is
$\mu$-strongly convex and $\mu$-smooth, while $f$ is $0$-smooth in the
sense of \citet{bubeck2015convex} as it is a linear function of $y$.
We see that the conjugate function \[ g^*(A y) := \max_{\eta}
\eta^{\top} Ay - g(\eta),\]
is therefore $\sigma$-strongly convex and $L$-smooth
with respect to $y$,
where $\sigma := \rho_{\min}^2\mu^{-1}$ and
$L := \rho_{\max}^2\mu^{-1}$.
With the definition of $g^*$, we also see that the problem considered
is equivalent to
\begin{equation}
	\min_{y \in \bY} f(y) + g^*(y).
	\label{eq:conjform}
\end{equation}
Since the objective \cref{eq:conjform} is strongly convex,
we know that there is a unique optimal solution $\bar y_{\mu}$.

We can easily see that the diameter of $\bY_i$ is upper bounded by
$\sqrt{m} \lambda_i$ from the constraints.
Therefore, it is clear that as long as $r_i - \eta_i \neq 0$ and we
take a step size
\[
	\theta_i \geq \frac{2 \sqrt{m} \lambda_i}{r_j - \eta_j},
\]
block coordinate gradient descent at the $i$-th block of $y$ gives the exact solution of the subproblem with respect to $y_i$.
On the other hand, if $r_i - \eta_i = 0$, any $y_i$ and thus any step size is optimal because the objective value remains a constant.
Therefore, from our assumption of the upper-boundedness of $\lambda_i/|r_j - \eta_j|$, our exact subproblem solution is equivalent to conducting the
project gradient descent at the selected blocks with any step size $\theta \geq C$.

Now we consider a ghost algorithm
\begin{equation}
	\label{eq:ghost}
	\tilde y^{t+1} = P_{\bY} \big( y^t - \theta S_t (\nabla f(y^t) + \nabla
	g^*(y^t)) \big),
\end{equation}
where $S_t$ is a diagonal matrix designating the blocks selected by our algorithm at the $t$-th iteration.
It is clear that
\begin{equation}
	\Ex[S_t] = \frac{B}{n} I, \quad \|S_t\| = 1.
	\label{eq:S}
\end{equation}
Therefore, by using the proof of Theorem~3.12 in \citet{bubeck2015convex} and the non-expansiveness of convex projections, we can easily obtain
\begin{equation}
	\Ex \|\tilde y^{t+1} - \bar y_{\mu}\|^2 \leq (1 -
	\frac{B\theta \sigma}{n}) \|y^t - \bar y_{\mu}\|^2,
	\label{eq:descent}
\end{equation}
provided that $\theta \leq 2B / (n(\sigma + L))$.

On the other hand, we can easily see that $\tilde y^{t+1} - y^{t+1} = \theta S_t A^{\top} ( \eta^t - \nabla g^*(Ay^t) )$, and therefore by
defining $C_1 := \sqrt{1 - B\theta \sigma / n}$,
\cref{eq:descent} implies that
\begin{align}
	\Ex \|y^{t+1} - \bar y_{\mu}\| 
    &= \Ex \| \tilde y^{t+1} - \bar y_{\mu}\| + \Ex \| y^{t+1} - \tilde y^{t+1} \| \nonumber\\
    &\leq C_1 \|y^t - \bar y_{\mu}\| + \theta \rho_{\max}
	\|\eta^t - \nabla g^*(Ay^t)\|.
	\label{eq:tobound}
\end{align}
With routine calculations, similar to Proposition~3.3
of \citet{du2019linear}, we also have that
\begin{equation}
	\Ex \|y^{t+1} - y^t\| \leq L \theta \|y^t - \bar y_{\mu}\| + \rho_{\max} \theta
	\|\eta^t - \nabla g^*(Ay^t)\|.
	\label{eq:prop33}
\end{equation}

Next, for the dual variables, we see that the update is over all coordinates without any stochastic elements.
Therefore, it can be seen as conducting gradient descent for the following objective function
\[
\tilde g(\eta) := g(\eta) - \eta^{\top} A y^t,
\]
and it is clear that $\tilde g$ is $\mu$-strongly convex and
$\mu$-smooth just like $g$.
Therefore, using \cref{eq:prop33} and again from the proof of Theorem~3.12 in \citet{bubeck2015convex}, for any $\tau \leq 1/\mu$, we can obtain after routine calculations that
\begin{align}
	\label{eq:dual}
	&\Ex\|\eta^{t+1} - \nabla g^*(Ay^{t+1})\|  \nonumber \\
    =& \Ex\|\eta^{t+1} - \nabla g^*(Ay^{t})\| + \Ex\| \nabla g^*(Ay^{t+1}) - \nabla g^*(Ay^{t})\| \nonumber \\
    \leq &  \underbrace{ \sqrt{1 - B \tau \mu/n} \| \eta^{t} - \nabla g^*(Ay^{t})\|  }_{ \text{Theorem~3.12 in \citet{bubeck2015convex}} } + \underbrace{ \frac{\rho_{\max}}{\mu} \| y^{t+1} - y^{t} \|  }_{ \text{Lipschitz Continuity} } \nonumber \\
    \leq & \sqrt{1 - B \tau \mu/n} \| \eta^{t} - \nabla g^*(Ay^{t})\| \nonumber \\ 
    & + \frac{\rho_{\max}}{\mu} \bigg(\underbrace{ L \theta \|y^t - \bar y_{\mu}\| + \rho_{\max} \theta
	\|\eta^t - \nabla g^*(Ay^t)\|  }_{ \eqref{eq:prop33} } \bigg)  \nonumber \\
    \leq & C_4	\|\eta^t - \nabla g^*(Ay^t)\| + L \theta \|y^t - \bar y_{\mu}\| , 
\end{align}
where $C_4 := \sqrt{1 - B \tau \mu / n} + L \theta$.

Summarizing \cref{eq:dual} and \cref{eq:tobound}, taking expectations
over all iterations, and by defining $a_t := \Ex\|y^t - \bar y_{\mu}\|$ and
$b_t := \Ex \|\eta^t - \nabla g^*(Ay^t)\|$, we have the following
recursion.
\begin{equation}
	\begin{cases}
		a_{t+1} &\leq C_1 a_t + C_2 b_t,\\
		b_{t+1} &\leq C_3 a_t + C_4 b_t,
	\end{cases}
	\label{eq:recursion}
\end{equation}
where $C_2 := \theta \rho_{\max}, C_3 := \rho_{\max} L \theta / \mu$.
From our assumptions and by setting $\theta = C$, we have that $C_1,
C_2, C_3, C_4 \in [0,1)$. Let us define $q \coloneqq \sqrt{\max\{C_1, C_2,
C_3, C_4\}}$, then clearly
\[
a_t^2 \leq a_t^2 + b_t^2 \leq q^t (a_0^2 + b_0^2)
= q^t \left(a_0^2 + b_0^2  \right).
\]
Finally, we just need to bound $a_0^2 + b_0^2$.
It is clear that $b_0^2 = \|\eta^0 - \mu^{-1} Ay^0 - \mu^{-1} c\|^2
\leq 3 (\|\eta^0\|^2 + \mu^{-2} \|c\|^2 + \mu^{-2} \|A\|^2 \|y^0\|^2)$
and $a_0$ is upper bounded by the diameter of $\bY$.
We therefore obtain the claimed result.
\hfill$\blacksquare$

\subsection{Proof of Theorem~\ref{Theo_2_simple}}\label{Proof:Theo_2_simple}

In Theorem~\ref{Theo_1}, we show that, under the existence of a constant $C$ and several conditions, Algorithm~\ref{alg:sda_non} converges linearly in the general case. 
In Theorem~\ref{Theo_2_simple}, by exploiting the special structure of the SBLP under consideration, we further simplify the conditions required in Theorem~\ref{Theo_1} and provide an explicit choice of the parameter $\mu$, ensuring linear convergence of the algorithm.

\begin{lemma}\label{Lemma_pre_1}
    The constraint matrix $A$ of the SBLP in standard form has full row rank.
\end{lemma}

\textit{Proof of Lemma~\ref{Lemma_pre_1}}.
Consider the SBLP in its standard form. 
To convert the inequality resource constraints~\eqref{entro_obj} into equalities, we introduce non-negative slack variables $s_j \geq 0$ for each $j \in [m]$. 
The constraints can then be written as $\sum_{i \in [n]} y_{ij} + s_j = c_j, \forall j \in [m]$.
Let the decision vector be defined as $\mathbf{y} = [y_{11}, \dots, y_{nm}, s_1, \dots, s_m]^\top \in \mathbb{R}^{m(n+1)}$. 
The constraint matrix $A \in \mathbb{R}^{m \times m(n+1)}$ then takes the following partitioned structure:
\[
A = \left[ 
\begin{array}{cccc|cccc}
\mathbf{1}^\top & \mathbf{0} & \dots & \mathbf{0} & 1 & 0 & \dots & 0 \\
\mathbf{0} & \mathbf{1}^\top & \dots & \mathbf{0} & 0 & 1 & \dots & 0 \\
\vdots & \vdots & \ddots & \vdots & \vdots & \vdots & \ddots & \vdots \\
\mathbf{0} & \mathbf{0} & \dots & \mathbf{1}^\top & 0 & 0 & \dots & 1
\end{array} 
\right] = [ M \mid I_m ]
\]
where $M \in \mathbb{R}^{m \times m (n+1)}$ represents the coefficients of the original variables $y_{ij}$ and $I_m \in \mathbb{R}^{m \times m}$ is the identity matrix corresponding to the slack variables. Since $A$ contains a full $m \times m$ identity matrix, its rows are linearly independent. Therefore, $\text{rank}(A) = m$, which implies that the matrix $A$ has full row rank.
\hfill$\blacksquare$

\begin{lemma}\label{Lemma_pre_2}
    The largest and smallest singular values of $A$, denoted by $\rho_{max}$ and $\rho_{min}$, can both be taken as 1.
\end{lemma}

\textit{Proof of Lemma~\ref{Lemma_pre_2}}.
Following Lemma A.1, the constraint matrix $A \in \mathbb{R}^{m \times ( n (m+1) )}$ is partitioned as $A = [M \mid I_m]$. 
Consequently, the Gram matrix $AA^\top$ is given by: $AA^\top = [M \mid I_m] \begin{bmatrix} M^\top \\ I_m \end{bmatrix} = MM^\top + I_m$.
Let $\Lambda(X)$ denote an eigenvalue of $X$. 
Since $MM^\top$ is positive semi-definite, $\Lambda_i(MM^\top) \ge 0$ for all $i$.
Thus, $\Lambda_i(AA^\top) = \Lambda_i(MM^\top) + 1 \ge 1$.
This implies $\rho_{\min} = \sqrt{\Lambda_{\min}(AA^\top)} \ge 1$. 
Following the standard normalization in SBLP, we take $\rho_{\min} = 1$.
Under the normalized structure of the simplex constraints in SBLP, the row sums of $M$ are restricted such that the spectral radius of $AA^\top$ is bounded by $1$. 
Therefore, $\rho_{\max} = \sqrt{\Lambda_{\max}(AA^\top)} = 1$, and $\rho_{\min} = \rho_{\max} = 1$.
\hfill$\blacksquare$

We simplify the various conditions appearing in Theorem~\ref{Theo_1}, where the requirements regarding the existence of constant $C$ are reduced to the most restrictive one $C < \mu \left( 1 - \sqrt{1 - \frac{B\tau \mu}{n}} \right)$.
Hence we get 
\begin{align}\label{eq_rela_1}
    \frac{1}{ \big| \sum\limits_{j \in [m]} ( r_j - \eta_j^t ) \big|} < \mu \left( 1 - \sqrt{1 - \frac{B \tau \mu}{n}} \right)
\end{align}
for any $t>0$.

\begin{lemma}\label{Lemma_pre_3}
    $\sum\limits_{j \in [m]} ( r_j - \eta_j^t ) \leq \sum\limits_{j \in [m]} w_{ij} r_j \big/ \sum\limits_{j \in [m]} w_{ij}$ for any $i \in [n]$.
\end{lemma}

\textit{Proof of Lemma~\ref{Lemma_pre_3}}.
We analyze how the stationarity condition of KKT couples the primal variables $y$ with the dual variables $\eta$.
We define the Lagrangian $L$ by incorporating the equality constraint (11) and the similarity inequality constraint (12):
$L = \sum\limits_{i \in [n]} \sum\limits_{j \in [m]} y_{ij} r_j + \sum\limits_{j \in [m]} \eta_j (c_j - \sum\limits_{i \in [n]} y_{ij}) + \sum\limits_{i \in [n]} \gamma_i (\lambda_i - \sum\limits_{j \in [m]} y_{ij}) + \sum\limits_{i \in [n]} \sum\limits_{j \in [m]} \phi_{ij} (\frac{y_{i0}}{w_{i0}} - \frac{y_{ij}}{w_{ij}})$.
The stationarity condition requires the partial derivative with respect to each $y_{ij}$ to be zero at optimality:
\begin{enumerate}
    \item For $j \in [m]$: $\frac{\partial L}{\partial y_{ij}} = r_j - \eta_j - \gamma_i - \frac{\phi_{ij}}{w_{ij}} = 0 \implies r_j - \eta_j = \gamma_i + \frac{\phi_{ij}}{w_{ij}}$.
    \item For the dummy variable $y_{i0}$: $\frac{\partial L}{\partial y_{i0}} = 0 - \gamma_i + \sum_{j=1}^m \frac{\phi_{ij}}{w_{i0}} = 0 \implies \gamma_i = \sum\limits_{j \in [m]} \frac{\phi_{ij}}{w_{i0}}$.
\end{enumerate}
By combining them, we obtain: $r_j - \eta_j = \sum\limits_{k = [m]} \frac{\phi_{ik}}{w_{i0}} + \frac{\phi_{ij}}{w_{ij}}$.
Summing this over all resources $j \in [m]$ for a fixed customer $i$:
\begin{align}
    \sum_{j \in [m]} w_{ij}(r_j - \eta_j) &= \sum_{j \in [m]} w_{ij} \left( \sum_{k=[m]} \frac{\phi_{ik}}{w_{i0}} + \frac{\phi_{ij}}{w_{ij}} \right)  \nonumber \\
    &= \left( \sum_{j \in [m]} w_{ij} \right) \left( \sum_{k \in [m]} \frac{\phi_{ik}}{w_{i0}} \right) + \sum_{j \in [m]} \phi_{ij}   \label{eq:1_1_1} 
\end{align}
Hence, it holds
\begin{align*}
    &\left| \sum\limits_{j \in [m]} (r_j - \eta_j) \right| \leq \frac{\sum\limits_{j \in [m]} w_{ij}(r_j - \eta_j)}{\sum\limits_{j=0}^m w_{ij}} = \frac{\eqref{eq:1_1_1}}{\sum\limits_{j=0}^m w_{ij}} \\
    &\leq \left( \sum\limits_{j \in [m]} w_{ij} \gamma_i + \sum\limits_{j \in [m]} \phi_{ij} \right) \Bigg/ \sum\limits_{j \in [m]} w_{ij} \\
        &= \max_{i\in [n]} \left\{\sum\limits_{j \in [m]} \left( w_{ij} \gamma_i + w_{ij} (r_j -\eta_j - \gamma_i) \right) \bigg/ \sum\limits_{j \in [m]} w_{ij} \right\}  \\
    &\leq  \max_{i\in [n]}  \left\{ \sum\limits_{j \in [m]} w_{ij} r_j  \bigg/ \sum\limits_{j \in [m]} w_{ij} \right\}.
\end{align*}
\hfill$\blacksquare$

\noindent
Let $R:= \max\limits_{i\in [n]}  \left\{ \sum\limits_{j \in [m]} w_{ij} r_j  \bigg/ \sum\limits_{j \in [m]} w_{ij} \right\}$, bring it into \eqref{eq_rela_1}. 
We obtain $\frac{1}{R^2 \mu^2} + \frac{B \tau \mu}{n} > \frac{2}{R \mu}$.
Setting $\mu = \frac{1}{2 R}$ yields a feasible solution.
\hfill$\blacksquare$

\subsection{Proof of Corollary~\ref{claim_1}}\label{Proof:claim_1}
The proof follows directly from Theorem~2.
In Algorithm~3, $k$ workers independently process mini-batches of size $B$ in parallel and aggregate their updates within each iteration.
As a result, each iteration effectively utilizes $kB$ sampled customers, while preserving the same stochastic structure and independence assumptions as in Algorithm~2.
Consequently, all arguments in the proof of Theorem~2 continue to hold with the batch size $B$ replaced by $kB$.
In particular, the contraction factor in the expected distance to the smoothed optimum improves from $1-\Theta(B/n)$ to $1-\Theta(kB/n)$, yielding the stated result.
\hfill$\blacksquare$

\subsection{Proof of Corollary~\ref{coro_new_2}}\label{Proof:coro_new_2}

The proof proceeds by verifying that the augmented matrix $\tilde{A}$ maintains full row rank, ensuring the dual problem is both strongly concave and smooth. 
Consequently, the linear convergence established in Theorem~\ref{Theo_1} applies, provided a valid constant $C$ exists, a condition equivalent to identifying a suitable $\mu$. 
Unlike the simplified case where $A$ is an identity-based structure with unit singular values, the bundle-selling matrix $\tilde{A}$ exhibits $\rho_{\max}, \rho_{\min} \neq 1$, necessitating a re-evaluation of the convergence parameters. 
By simplifying the bounds from Theorem~\ref{Theo_1}, we determine a value for $\mu$ to be $\rho_{\max}^2 / 2\tilde{R}$.

We describe the augmented constraint matrix $\tilde{A}$, which facilitates the transformation of the inequality constraints into a system of linear equations. 
This is achieved by introducing a vector of non-negative slack variables $\mathbf{s} = [\mathbf{s}_{res}^\top, \mathbf{s}_{ratio}^\top]^\top$, where $\mathbf{s}_{res} \in \mathbb{R}^K$ and $\mathbf{s}_{ratio} \in \mathbb{R}^{nm}$ account for the unused resource capacities and the satisfaction of the selection logic, respectively. 
The resulting block matrix $\tilde{A}$ is partitioned as follows:
$$\tilde{A} = 
\left[
\begin{array}{cccc|c}
\tilde{B} & \tilde{B} & \dots & \tilde{B} & \mathbf{I}_{L \times L} \\ \hline
\mathbf{1}^\top & \mathbf{0} & \dots & \mathbf{0} & \mathbf{0} \\
\mathbf{0} & \mathbf{1}^\top & \dots & \mathbf{0} & \mathbf{0} \\
\vdots & \vdots & \ddots & \vdots & \vdots \\
\mathbf{0} & \mathbf{0} & \dots & \mathbf{1}^\top & \mathbf{0} \\ \hline
\mathbf{W}_1 & \mathbf{0} & \dots & \mathbf{0} & \mathbf{0} \\
\mathbf{0} & \mathbf{W}_2 & \dots & \mathbf{0} & \mathbf{0} \\
\vdots & \vdots & \ddots & \vdots & \vdots \\
\mathbf{0} & \mathbf{0} & \dots & \mathbf{W}_n & \mathbf{I}_{ratio}
\end{array}
\right].$$
To align with the augmented matrix $\tilde{A}$, we define the complete decision vector as $\mathbf{x} = [\mathbf{y}_1^\top, \dots, \mathbf{y}_n^\top, \mathbf{s}_{res}^\top, \mathbf{s}_{ratio}^\top]^\top$.
Here, $\mathbf{y}_i = [y_{i,0}, \dots, y_{i,m}]^\top$ represents the selection probabilities for customer $i$. The matrix is partitioned into three functional layers:
(i) Resource Layer: Implements $\sum_{i,j} \tilde{B}_{\ell j} y_{ij} \le c_\ell$ by repeating the incidence matrix $\tilde{B}$ across all customers. The identity block $\mathbf{I}_{L \times L}$ maps to slack variables $\mathbf{s}_{res} \in \mathbb{R}^L$, capturing residual product capacities.
(ii) Demand Layer: Enforces demand conservation for each customer ($\sum_j y_{ij} = \lambda_i$). As these are equality constraints, no slack variables are required.
(iii) Selection Logic Layer: Incorporates proportional constraints $y_{ij}/w_{ij} \le y_{i0}/w_{i0}$. The identity block $\mathbf{I}_{ratio}$ corresponds to slack variables $\mathbf{s}_{ratio} \in \mathbb{R}^{nm}$.
This construction ensures that $A_{total}$ maintains full row rank.
Following the proofs of Lemma~\ref{Lemma_pre_1} and Lemma~\ref{Lemma_pre_2}, we can see that $\rho_{\max}>1$ and $\rho_{\max} > \rho_{\min} > 0$.

By simplifying the conditions in Theorem~\ref{Theo_1} and isolating the most restrictive bound, we obtain:
$$\frac{1}{\tilde{R}} < \left( 1 - \sqrt{1 - \frac{B \tau \mu}{n} }\right) \frac{\mu}{\rho_{\max}^2}.$$
Completing the algebraic simplification of this inequality concludes the proof.
\hfill$\blacksquare$

\clearpage

\section{Extended Analysis and Details of the Case Study with ZOZOTOWN Data}\label{App_2}
This section provides supplementary information and detailed methodological breakdowns for the empirical case study discussed in Section~\ref{sec6:case}. 
The primary goal is to ensure the reproducibility and transparency of our findings. 
We first detail the complete data preprocessing pipeline and the feature engineering used for the MNL preference estimation (including customer clustering and product attribute definition). 
Following this, we provide an extended analysis and additional performance visualizations to further strengthen the conclusions drawn in the main paper regarding the comparison between the Global Optimization (GO) and Market Segment Decentralized (MSD) approaches.

\subsection{Supplementary Empirical Context}\label{SecAPP:B1}
Figure~\ref{Fig:APillus_1} illustrates the operational environment of the ZOZOTOWN recommendation interface from which our empirical data are constructed.
When a customer visits the detail page of a particular SKU (highlighted in the blue dashed box), the platform simultaneously displays a set of recommended items directly beneath the focal product (red dashed box). 
These recommended items constitute the choice set from which the customer may potentially select an alternative item.
Each recommended product card includes the item’s image, brand name, and the discounted price currently offered.
The recommendation panel typically contains 8 items per page, and customers can scroll to view additional recommended products depending on the category.
\begin{figure}[tbp]
\caption{\raggedright Illustration of the ZOZOTOWN product-detail page showing the focal item viewed by the customer, the recommended items displayed beneath it, and the real-time inventory status for each size variant. 
Only items with positive inventory are shown as feasible recommendations. }
    \centering
\includegraphics[width=125mm]{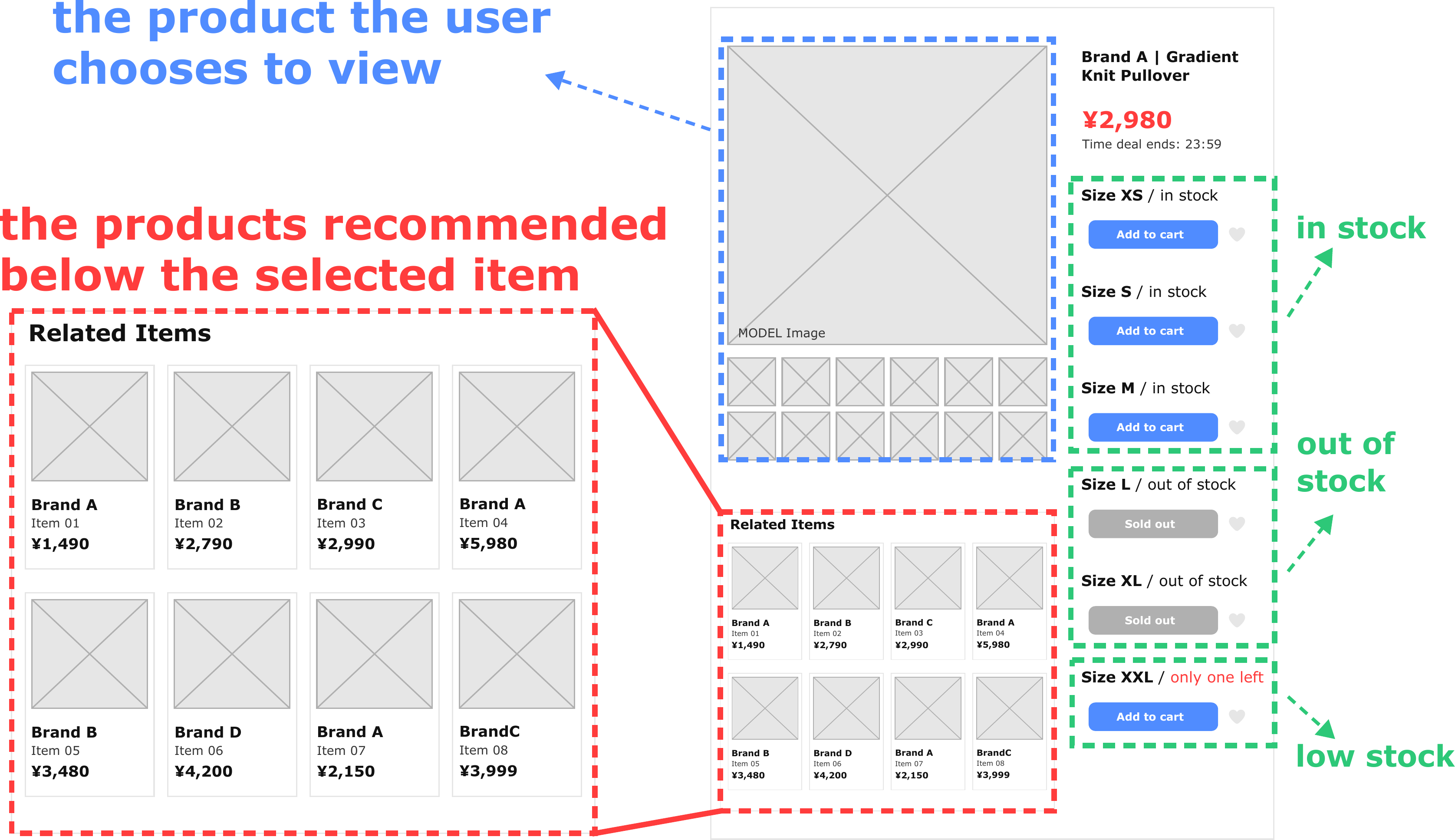}
\label{Fig:APillus_1}
\end{figure}
On the right-hand side of the product page, ZOZOTOWN displays the available inventory information for each size variant of the focal product (green dashed box).
The inventory status is categorized into three levels:
\begin{itemize}
    \item In stock (sufficient inventory; blue “Add to Cart” button),
    \item Low stock (inventory almost depleted; displayed with a warning label), and
    \item Out of stock (greyed-out button and labeled as sold out).
\end{itemize}
ZOZOTOWN displays only recommended products with strictly positive inventory, so both “in-stock’’ and “low-stock’’ SKUs remain visible to users while out-of-stock items are excluded. 
As a result, customers are always presented with a feasible choice set, and a click followed by an add-to-cart action is treated as the realized choice.
For each browsing session, we record the focal product viewed, the set of recommended items displayed at that moment, and the user’s subsequent click and add-to-cart behavior. 
For every recommended SKU, we also capture its price and real-time inventory level at exposure. 
In addition, we extract product attributes, such as color, to construct feature vectors for estimating user preference patterns. 
These data allow us to observe realized choices from feasible assortments while incorporating the product characteristics required to model user–SKU preferences.

\subsection{Supplementary Data Information}
The empirical context described in Section~\ref{SecAPP:B1} gives rise to a rich dataset capturing user interactions with product pages and recommendation panels on ZOZOTOWN. 
In our empirical analysis, we focus on a single day in June 2024 and track the behavior of a subset of active users on the platform. 
For each user, we observe the sequence of clicked products together with the corresponding SKU identifiers, product attributes such as color, and the displayed prices at the moment of interaction. 
For every SKU appearing in the dataset, we also collect its initial inventory level at the start of the day, abstracting away from any potential restocking during the observation period. Each interaction is time-stamped, enabling us to reconstruct the temporal order of browsing and choice activities. 
These components jointly form the dataset used for preference estimation and the subsequent case study analysis, as summarized in Table~\ref{TabApp:Data_1}.

\begin{table}[tbp]
\centering
\small
\caption{
Summary of Raw Data and Processed Dataset Used in the Case Study
}
\renewcommand*{\arraystretch}{1.2}
\begin{tabular}{c|ccccccc}
\toprule
 & \#purchases & \#customers & \#SKUs & \#colors & \makecell{price range \\ (mean)} & \makecell{stock range \\ (mean)} & time-spam \\ \midrule
Raw & \num{5342966} & \num{64700} & \num{49526} & \num{239} & - & - & - \\
Processed & \num{1372671} & \num{38746} & \num{1000} & \num{145} & \makecell{[100, $2.58 \mathrm{e}5$] \\ (4505)} & \makecell{[1, $3 \mathrm{e}8$] \\ (999)}   & 23:59:31 \\
\bottomrule
\end{tabular}
\label{TabApp:Data_1}
\end{table}

The initial dataset contained $\text{5,342,966}$ raw purchase interactions, which was strategically refined to a processed dataset containing $\text{1,372,671}$ records. 
This significant filtering process similarly reduced the customer pool from $\text{64,700}$ to $\text{38,746}$ active users. 
The reduction in record count and users is largely attributable to the strict temporal scope applied, confining the observation window to precisely $\text{23:59:31}$. 
Despite the reduction, the final processed volume remains substantial, supporting a robust empirical analysis. 
Along the product dimension, strategic filtering was rigorously applied: the total number of unique $\text{SKUs}$ was drastically reduced from $\text{49,526}$ to $\text{1,000}$, and the diversity of product colors saw a parallel refinement, dropping from $\text{239}$ to $\text{145}$. 
This focused approach ensures that the analysis and subsequent models concentrate exclusively on the most representative and frequently interacted products, thereby simplifying computational complexity while maintaining coverage of core consumer choice behavior. 
Furthermore, clear numerical boundaries were established for key variables: the price range was confined to $[\text{100}, \text{2.58} \times \text{10}^5]$, displaying an average price of $\text{4,505}$. 
The price distribution visualization confirms this finding, showing the mean price to be $\text{4,406}$ and the median price at $\text{3,635}$, a pattern indicative of a strong positive skew where the majority of products are priced toward the lower end of the range. 
Concurrently, the initial inventory level (stock range) was set from $[\text{1}, \text{3} \times \text{10}^8]$, with a notably lower mean initial inventory of $\text{999}$. 
The explicit lower bound of $\text{1}$ confirms that all $\text{1,000}$ products included in the analysis were verifiably available in stock at the moment of user interaction, a critical prerequisite for accurate demand modeling and inventory analysis.

\subsection{Estimating Customer–SKU Preference Weights with an MNL–MLE Framework}

The adopted framework employs a Segmented Multinomial Logit (MNL) model estimated via Maximum Likelihood Estimation (MLE) to capture unobserved consumer heterogeneity in product preferences. 
This approach addresses the limitation of a standard MNL model, which assumes homogeneous preferences across all users. 
The core logic follows a two-stage process: 
First, customers are clustered into homogeneous groups based on their historical behavior using K-Means; 
Second, a unique set of utility weights ($\beta_k$) is estimated for each segment using the MNL-MLE framework.

\subsubsection{Stage One: Customer Segmentation via K-Means Clustering}

The initial stage involves grouping the total customer base of $\text{38,746}$ users into $K$ distinct preference segments.

\textbf{Feature Engineering for Clustering.}
The customer feature matrix ($\mathbf{X}_{\text{customer}}$) is constructed from historical purchasing data to represent underlying preference traits:
\begin{itemize}
    \item \textit{Price Sensitivity Features}: 
    These include the mean purchase price and the standard deviation of purchase prices, which quantify a customer's spending level and price consistency.
    \item \textit{Color Preference Features}:
    These are represented by the historical frequency of purchasing products of a specific color $i$, indicating strong aesthetic biases.
\end{itemize}
All constructed features undergo Standardization to ensure that all dimensions contribute equally to the Euclidean distance calculation during clustering:$$\mathbf{X}_{\text{scaled}} = \frac{\mathbf{X} - \mu_{\mathbf{X}}}{\sigma_{\mathbf{X}}}$$where $\mu_{\mathbf{X}}$ and $\sigma_{\mathbf{X}}$ are the mean and standard deviation of the feature matrix $\mathbf{X}$.

\textbf{Optimal Cluster Selection and Iterative Refinement.}
The initial number of clusters, $K$, is determined using the Elbow Method, minimizing the Within-Cluster Sum of Squares ($\text{WCSS}$):$$\text{WCSS}(K) = \sum_{k=1}^{K} \sum_{\mathbf{x} \in C_k} ||\mathbf{x} - \mu_k||^2$$where $C_k$ is the $k$-th cluster and $\mu_k$ is its centroid.
Following initial clustering, an Enhanced Iterative Cluster Refinement process is applied to enforce a minimum sample size constraint required for the subsequent MLE estimation.
A segment is deemed viable only if its total number of purchase records ($\text{N}_{\text{records}}$) exceeds a strict training threshold ($\text{MIN\_RECORDS\_CHECK}$):$$\text{MIN\_RECORDS\_CHECK} = N_{\text{params}} + 5$$where $N_{\text{params}}$ is the total number of parameters (product features plus the intercept term) to be estimated in the MNL model. 
This iterative process dynamically recalculates the centroids of viable clusters and reassigns all customers from undersized ``small clusters'' to the nearest viable cluster, ensuring a final set of $K_{\text{effective}}$ segments ready for robust estimation.

\begin{figure}[tbp]
\caption{\raggedright Elbow Method for Determining the Optimal Number of Clusters ($K$). }
    \centering
\includegraphics[width=125mm]{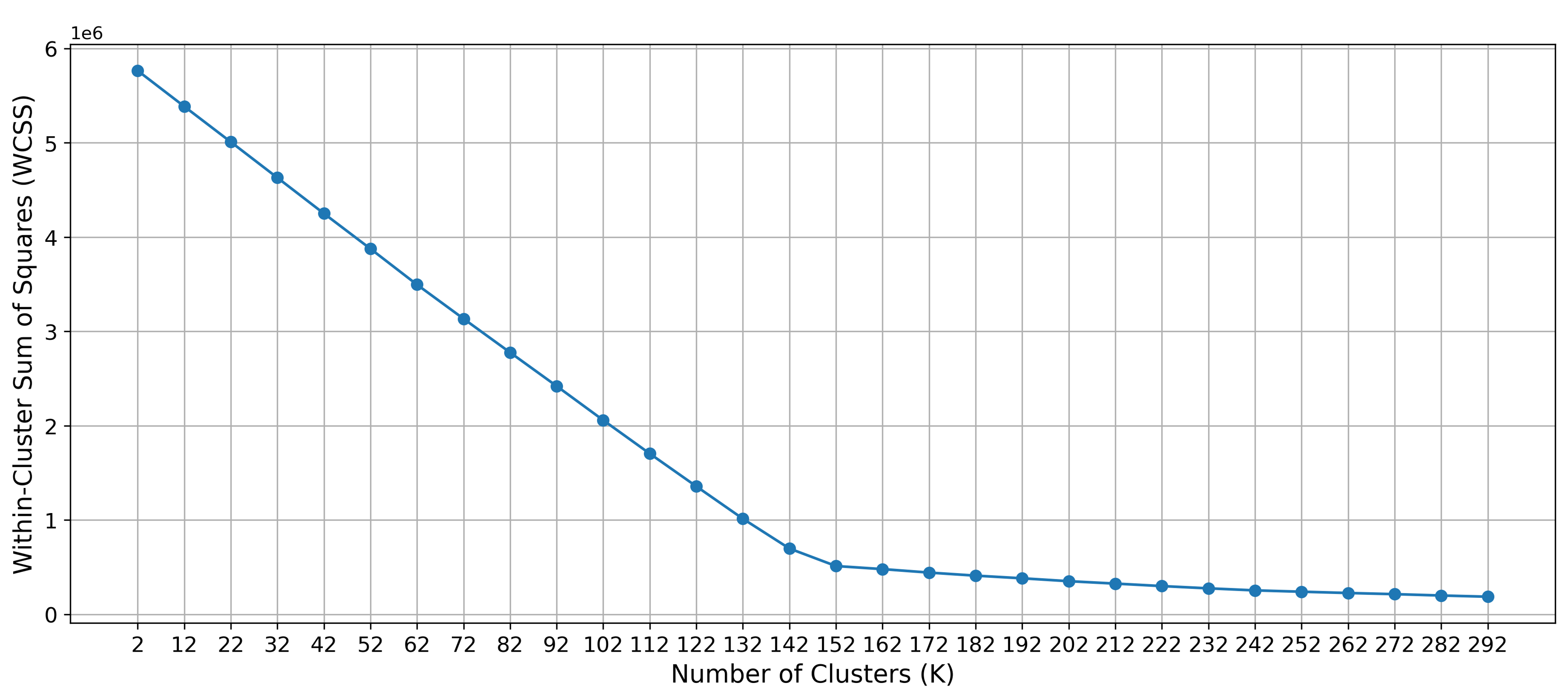}
\label{Fig:APillus_2}
\end{figure}

As illustrated in Figure~\ref{Fig:APillus_2}, the Within-Cluster Sum of Squares (WCSS) decreases sharply as $K$ increases from 2 up to approximately 102. 
The rate of reduction begins to slow significantly between $K=102$ and $K=152$, marking the region of the ``elbow''.
Beyond $K \approx 160$, the curve flattens out, indicating that increasing the cluster count provides minimal additional explanatory power for the variance in customer features. 
Based on this analysis, and balancing the goal of capturing rich heterogeneity with the practical need for computational tractability, we selected $K=160$ as the optimal number of segments. 
This segmentation method allowed us to partition the \num{38746} customers into 160 distinct groups, each representing a homogeneous set of preference patterns.

\subsubsection{Stage Two: Segmented MNL-MLE Estimation}

For each of the $K_{\text{effective}}$ segments, a separate MNL model is estimated.

\textbf{The Multinomial Logit Model (MNL).}
The MNL model is rooted in the assumption that customer $i$ chooses the option $j$ that maximizes their utility $U_{ij}$. 
The utility function for product $j$ within segment $k$ is specified as a linear combination of its observable features $\mathbf{X}_j$:
$$V_{ij} = \beta_{k, 0} + \sum_{m=1}^{M} \beta_{k, m} X_{j, m} = \mathbf{X}_j \cdot \mathbf{\beta}_k$$
where the product features $\mathbf{X}_j$ include the standardized price and the color One-Hot Encoding ($\text{OHE}$).
Assuming the error terms are $\text{IID}$ Gumbel distributed, the probability $P_{ij}$ that a customer $i$ in segment $k$ chooses option $j$ from the choice set $C_i$ is given by:
$$P_{ij}(\mathbf{\beta}_k) = \frac{\exp(V_{ij})}{\sum_{l \in C_i} \exp(V_{il})}$$
The choice set $C_i$ is constructed using a sampling method, consisting of the actual chosen product ($j^*$) and $M_{\text{alternatives}} = 5$ randomly sampled non-chosen products.

\textbf{Maximum Likelihood Estimation (MLE) and Optimization.}
The segment-specific utility weights $\mathbf{\beta}_k$ are determined by minimizing the Negative Log-Likelihood (NLL) function, which corresponds to maximizing the likelihood of observing the actual choices made by customers in that segment.
The $\text{NLL}$ function for segment $k$ is defined as:
$$\text{NLL}(\mathbf{\beta}_k) = -\sum_{i \in \text{Cluster } k} \ln(P_{i, j_i^*}(\mathbf{\beta}_k)).$$
Substituting the MNL probability formula yields the expression to be minimized:
$$\text{NLL}(\mathbf{\beta}_k) = -\sum_{i \in \text{Cluster } k} \left[ V_{i, j_i^*} - \ln\left(\sum_{l \in C_i} \exp(V_{il})\right) \right].$$
The minimization is performed using the robust L-BFGS-B optimization algorithm, a quasi-Newton method suitable for high-dimensional parameter estimation.

\textbf{Ensemble Averaging.}
To enhance the stability and robustness of the estimated parameters against the randomness introduced by the sampled choice sets, the entire estimation process is repeated $N_{\text{ENSEMBLE\_RUNS}} = 3$ times, each with a different random seed. 
The final estimated utility weight vector $\hat{\beta}_k$ for each segment is computed as the average of the parameters obtained from all successful runs:
$$\hat{\beta}_k = \frac{1}{N_{\text{success}}} \sum_{r=1}^{N_{\text{success}}} \hat{\beta}_{k}^{(r)}.$$
This averaging strategy mitigates variance and provides a more reliable estimation of consumers' latent preferences.

\subsection{Supplementary Results for the Comparison Between GO and MSD}\label{App_secB4}

\begin{table}[tbp]
\centering
\small
\caption{Comparison of Key Performance and Mechanism Metrics between GO and MSD Policies}
\renewcommand*{\arraystretch}{1.2}
\label{tab:mechanism_comparison}
\begin{tabular}{lccc}
\toprule
\textbf{Metric} & \textbf{GO} & \textbf{MSD} & \textbf{Difference}  \\
\midrule
\textbf{Product Sales Diversity (Shannon Entropy)} & $6.4$ & $9.4$ & GO $<$ MSD  \\ 
\textbf{Mean Opportunity Cost ($\overline{v}$)} & $100.1$ & $0$ & GO $\gg$ MSD \\
\bottomrule
\end{tabular}
\end{table}

The performance disparity between the Global Optimization (GO) and Market Segment Decentralized (MSD) policies is fundamentally rooted in their resource pricing mechanisms, as evidenced by two key metrics.

\textbf{Analysis of Resource Scarcity Pricing and Allocation Efficiency.}
The Product Sales Diversity Index is quantified using Shannon Entropy ($H$), a foundational concept in information theory that measures the degree of uncertainty or randomness within a probability distribution. 
In the context of this study, $H$ is applied to the distribution of successful transactions across all distinct product types to assess the uniformity of resource allocation. 
A higher $H$ value indicates that sales are broadly distributed across a greater variety of products (high diversity), while a lower $H$ value suggests that sales are highly concentrated on a few specific products (high centralization). 
The index is calculated using the formula: 
$$H = - \sum_{i=1}^{N} P_i \log_2(P_i),$$ 
where $N$ is the total number of distinct products sold, and $P_i$ is the relative frequency (probability) of product $i$'s sales quantity within the total sales volume, ensuring that $\sum P_i = 1$. The base-2 logarithm ($\log_2$) standardizes the unit of measurement to bits.
The resulting entropy values, MSD ($9.40$) and GO ($6.39$), are critical for characterizing the allocation strategies. 
The significantly lower entropy of GO demonstrates that its unified high scarcity threshold ($\overline{\lambda} \approx 100.1$) successfully forces resource competition, leading to a strategic concentration of sales on the most profitable products. 
Conversely, the higher entropy of MSD is a direct consequence of its zero allocation threshold ($\overline{\lambda} \approx 0$), which permits a wide dispersion of transactions across a greater variety of items, including low-value long-tail products, proving its tendency towards allocation diversity at the expense of maximized total revenue.

\textbf{Structural Failure of MSD: Zero Opportunity Cost and Amplified Stockouts.}
Our experimental finding that the MSD policy operates at an average shadow price of approximately zero provides a structural explanation for the frequent stockout phenomenon observed in large-scale e-commerce platforms with decentralized recommendation mechanisms. 
A near-zero shadow price implies that MSD effectively assigns no opportunity cost to scarce inventory, thereby operating under what is essentially a \emph{Zero Allocation Barrier}. 
Under this condition, any transaction with positive marginal revenue $r_j - \lambda_j > 0$ (with $\lambda_j \approx 0$) is locally accepted and recommended. 
Consequently, virtually all feasible transactions are treated as profitable within each segment, regardless of their relative value from a global perspective.
This mechanism leads to a profound coordination failure. 
Because recommendation decisions are made independently across multiple segments, each segment may simultaneously recommend the same scarce SKU based on outdated or inconsistent inventory signals. 
This decentralization produces systematic over-commitment: the same unit of inventory is implicitly promised to multiple customers at once. 
As a result, globally limited stock is rapidly depleted by numerous low-value local transactions before high-value customers ever encounter it. 
The absence of a unified opportunity-cost filter, such as the high global shadow price $\overline{\lambda} \approx 100.1$ generated under the GO policy, prevents the system from prioritizing transactions that produce the greatest economic value.
From the customer perspective, this structural deficiency manifests as a recurring commercial failure mode: an item appears available and is recommended, yet the checkout attempt ultimately fails because the inventory has already been consumed by other segments. 
The resulting ``out of stock'' messages are not random operational glitches but a predictable consequence of MSD's inability to internalize global scarcity. 
In essence, MSD's zero shadow-price regime creates aggressive and uncoordinated competition for scarce inventory, systematically eroding both platform revenue and user experience.
These observations highlight that decentralized recommendation architectures require a meaningful mechanism for pricing global scarcity. 
Without such a mechanism, the platform inevitably misallocates inventory toward low-value consumption, amplifies stockout frequency, and undermines the allocation efficiency that global optimization-based policies are designed to protect.

\end{document}